\documentclass{amsart}

\newtheorem{thm}{Theorem}[section]
\newtheorem{lem}[thm]{Lemma}
\newtheorem{defn}[thm]{Definition}
\newtheorem{rmk}[thm]{Remark}
\newtheorem{prop}[thm]{Proposition}
\newtheorem{cor}[thm]{Corollary}
\newenvironment{claimm}[1]{\par\noindent\textbf{Claim:}\space#1}{}
\newenvironment{claimmproof}[1]{\par\noindent{Proof:}\space#1}{}

\title{A Type B Analogue to Ribbon Tableaux}

\usepackage[margin=1in]{geometry} 
\usepackage[utf8]{inputenc}
\usepackage{amsfonts}
\usepackage{dsfont} 
\usepackage{tikz}
\usepackage{ytableau}    
\usepackage{xcolor}
\usepackage{float}
\usepackage{bbding}
\usepackage{amsmath}
\usepackage{amsthm}
\usepackage{amssymb}

\begin{document}

\author{Ezgi Kantarc\i\ O\u{g}uz}

\address{Department of Mathematics\\ University of Southern California \\ Los Angeles, CA}
\email{kantarci@usc.edu}

\begin{abstract}We introduce a shifted analogue of the ribbon tableaux defined by James and Kerber \cite{MR644144}. For any positive integer $k$, we give a bijection between the $k$-ribbon fillings of a shifted shape and regular fillings of a $\lfloor k/2\rfloor$-tuple of shapes called its $k$-quotient. We define the corresponding generating functions, and prove that they are symmetric, Schur positive and Schur Q-positive. Then we introduce a Schur Q-positive $q$-refinement. 
\end{abstract}
\maketitle
\section{Introduction}
The study of ribbon tableaux on shifted shapes combines two existing areas of work: the theory of ribbon tableaux and Schur's Q-functions. Ribbon tableaux introduced by James and Kerber\cite{MR644144} have applications to the representations of the symmetric group over a field of finite characteristic. Their theory was extend to the LLT polynomials by Lascoux, Leclerc and Thibon which arise in the Fock space representation of the universal enveloping algebra of quantum affine $\mathfrak{sl}_n$\cite{MR1434225}.  An expansion of Macdonald polynomials in terms of the LLT polynomials is given in \cite{MR2138143}, and many other important symmetric functions have natural expansions into LLT polynomials.

Schur's Q-functions come up as the symmetric functions that correspond to the shifted diagrams. They have a connection to the irreducible spin characters of the symmetric group, analogous to that of Schur functions and irreducible characters of linear representations\cite{MR3443860}. Since their introduction in \cite{MR1580818}, applications to diverse mathematical fields have been discovered, including the
cohomology of isotropic Grassmannians \cite{MR1094746} and polynomial solutions to the BKP equation in hydrodynamics.

In this work, we are merging these two ideas to initiate a combinatorial theory of ribbon tableaux for shifted shapes. The $k$-quotients and $k$-cores for shifted shapes were previously studied by Morris and Yaseen in 1986\cite{MR809494}. We expand upon their work, reformulating it in a more explicit way that is analogous to the ribbon tilings of unshifted shapes due to James and Kerber\cite{MR644144}. We also look at standard and semi-standard fillings of these shapes, and define shifted $k$-ribbon functions. We give a positive expansion in terms of Schur's Q-functions, analogous to the unshifted case.

The positivity result hints at the possibility of defining a type B analogue for the LLT polynomials which could have far reaching applications. We give such a definition using Type $A$ LLT polynomials and prove its Schur Q-positivity. We further show that there is no natural expansion of the spin statistic that would allow a direct definition using shifted ribbon tableaux, and provide some counter examples that should prove valuable for further research.  

The layout of this paper is as follows: In Section 2, we recall the notions of Schur functions, Schur's Q-functions and ribbon tableaux. In Section 3, we give a graphical description of $k$-ribbons on a shifted diagram, which differs from the standard case in that we have some 'double ribbons', which are allowed to contain $2\times2$ boxes. We define the shifted $k$-ribbon tableaux and the corresponding P- and Q-functions, as well as state our main theorem giving an expansion of a shifted ribbon Q-function in terms of Schur's Q= funtions. Sections 4 and 5 give the combinatorial constructions necessary to prove this, including a new type of object that comes up in shifted $k$-quotients which we call folded tableaux. We give bijections between ribbon fillings of a shifted diagram and its $k$-quotient, both in the standard and semi-standard case. In Section 5, we give a description of the shifted ribbon  functions in terms of peak functions. Lastly, in Section 7, we define a $q$-refinement of the shifted ribbon function and prove its Schur Q-positivity. We further discuss the difficulties of defining a direct analogue of the spin statistic from the unshifted case, and provide some counter-examples. 

\section{Preliminaries}

\subsection{Schur Functions}
A \emph{partition} of $n$ is a list $\mu=(\mu_1,\mu_2,\ldots,\mu_k)$ of non-increasing positive integers adding up to $n$, called its \emph{parts}. Here, $n$ is called the \emph{size} of the partition, denoted $|\mu|$, and the number of its parts is called its \emph{height} denoted ht($\mu$). With every partition, we associate a \emph{Young diagram}, an array with $\mu_i$ boxes on row $i$. 

A \emph{semi-standard Young tableau} of shape $\mu$ is a filling of its boxes with positive integers such that each column will be increasing from bottom to top, and each row will be non-decreasing from left to right. A semi-standard tableau that contains each of the numbers from $1$ to $n$ exactly once is called \emph{standard}. We will denote the set of semi-standard tableaux of shape $\mu$ by $SSYT(\mu)$, and the set of standard ones by $SYT(\mu)$. 

\begin{figure}[ht]
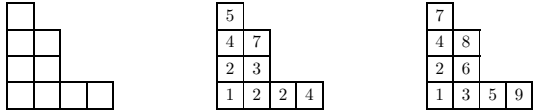

\center
\scalebox{.65}{
     \begin{ytableau}
   \\
    &  \\
   &  \\
    & & & \\
  \end{ytableau}\qquad \qquad \qquad  \begin{ytableau}
   5\\
  4 &7  \\
  2 & 3\\
    1 &2 &2 &4 \\
  \end{ytableau}\qquad \qquad \qquad  \begin{ytableau}
   7\\
    4&8\\
   2& 6\\
    1& 3& 5& 9\\
  \end{ytableau}}
  \caption {The diagram $\mu=(4,2,2,1)$ with corresponding semi-standard and standard fillings.}
  \label{fig:startunshifted}
\end{figure}
For a partition $\mu$ we define its \emph{Schur function} as follows:
\begin{equation}
s_{\mu}(X)=\sum_{T\in SSYT(\mu)}X^{T}
\end{equation}
Here $X^{T}$ denotes the monomial where the power of $x_i$ is given by the number of times $i$ occurs in $T$. The semi-standard filling in Figure \ref{fig:startunshifted} corresponds to the monomial $x_1x_2^3x_3x_4^2x_5x_7$.

The \emph{reading word} of a tableau is a reading of all its labels from left to right, top to bottom.
For example, the semi-standard tableau from Figure \ref{fig:startunshifted} has the reading word $547231224$, where as the standard one has the reading word $748261359$.

Note that the reading word of a standard tableau $S$ gives a permutation of numbers from $1$ to $|n|$, so we can talk about its descent, peak and spike sets. The \emph{descent set} of a standard tableau $T$ is defined as follows:

\begin{eqnarray*}
\text{Des}(T)&=&\{i\mid i\text{ is to the left of }i+1\text{ in the reading word of }T\}\subset [n-1]
\end{eqnarray*}

In general, for any set $D \subset [n]$, the peak and spike sets of $D$ are given by: 

\begin{eqnarray*}
\text{Peak}(D)&=&\{i\mid i\in \text{D}\text{ and }i-1 \notin \text{D}\}\\
\text{Spike}(D)&=&\{i\mid i\in \text{D}\text{ and }i-1 \notin \text{D}\text{ or }i\notin \text{D}\text{ and }i-1 \in \text{D}\}.
\end{eqnarray*}

Throughout this work, we will mainly be interested in the case when $D$ is the descent set of the reading word for a tableau. For a tableau $T$, we will use to notations $Peak(T)$ and $Spike(T)$ to denote $Peak(Des(T))$ and $Spike(Des(T))$ respectively. The standard tableau from Figure \ref{fig:startunshifted} has descent, peak and spike sets $\{1,3,5,6\}$, $\{3,5\}$ and $\{2,3,4,5,7\}$ respectively. In 1984, Gessel\cite{MR777705} has shown that the Schur function for a partition $\mu$ can be expressed in terms of descent sets:
\begin{equation*}
s_{\mu}(X)=\sum_{T\in SYT(\mu)}F_{Des(T)}(X)\end{equation*}
where $F_{D}(X)$, $D\in[n-1]$ denotes Gessel's fundamental basis for quasisymmetric functions  defined by:
\begin{equation}
F_D(X)=\sum_{\substack{i_1\leq i_2\leq \cdots \leq i_m\\t\in D \Rightarrow i_t \neq i_{t+1}}}x_{i_1}x_{i_2}x_{i_3}\cdots x_{i_m}\end{equation}
This formula allows us to calculate the Schur function of a partition using only its standard fillings. For example, the Schur function of (3,2), whose standard fillings are given in Figure \ref{fig:32schur} is:
$$s_{(3,2)}(X)=F_3(X)+F_2(X)+F_4(X)+2F_{2,4}(X)$$
\begin{figure}[ht]
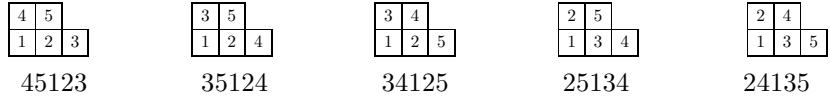

    \centering
    \scalebox{.65}{
     \begin{ytableau}
   4&5  \\
    1& 2&3\\
  \end{ytableau}\qquad \qquad \qquad  \begin{ytableau}
   3& 5 \\
    1&2 &4\\
  \end{ytableau}\qquad \qquad \qquad  \begin{ytableau}
   3& 4 \\
    1&2 &5\\
  \end{ytableau}\qquad \qquad \qquad  \begin{ytableau}
   2& 5 \\
    1&3 &4\\
  \end{ytableau}
  \qquad \qquad \qquad  \begin{ytableau}
   2& 4 \\
    1&3 &5\\
  \end{ytableau}}
  \\
  \vspace{0.2cm}
  45123 \qquad \quad \quad 35124  \quad \qquad \quad 34125 \quad \qquad \quad 25134 \quad \quad \qquad 24135
\caption{The standard tableaux of shape $(3,2)$ and their reading words}
\label{fig:32schur}
\end{figure}

\subsection{Ribbon Tableaux}

Given two diagrams $\mu \subset \nu$, the \emph{skew diagram} $\nu / \mu$ is the diagram of  $\nu$ minus the cells that correspond to $\mu$. A \emph{$k$-ribbon} on an unshifted diagram is a connected skew-diagram that contains no $2\times2$ square. A $k$-ribbon $R$  is removable from diagram $\mu$ if $\mu / \nu=R$ for some $\nu \subset \mu$. A diagram with no removable $k$-ribbon is called a $k$-core. 

 On a given {$k$-ribbon} $R$, the rightmost lowest cell is called the \emph{head} of the $R$. A set of disjoint ribbons form a \emph{horizontal strip} if their disjoint union is a (not necessarily connected) skew-shape and their heads lie on different columns. A \emph{semi-standard $k$-ribbon tableau} of shape $\mu$ is a sequence of shifted diagrams $\mu_0\subset \mu_1\subset\cdots\subset \mu_n=\mu$ where $\mu_0$ is a $k$-core, and each $\mu_{i}/ \mu_{i-1}$ is a horizontal $k$-ribbon strip, the ribbon on which we label by $i$.  A semi-standard $k$-ribbon tableau is called \emph{standard} if all labels from $1$ to $n$ occur exactly once. The generating function for the $k$-ribbon tableaux of shape $\mu$ is given by:

$$GF_{\mu/\mu_0}^{(k)}(X)=\sum_{T \in SSRT_k(\mu)}X^T=\sum_{S \in SRT_k(\mu)}F_{Des(S)}X$$ where $SSRT_k(\mu)$ denotes the set of semi-standard $k$-ribbon tableaux of shape $\mu$, and $SRT_k(\mu)$ denotes the set of standard ones. 

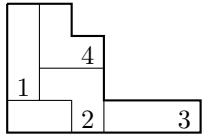
\begin{figure}
    \centering
\begin{tikzpicture}[scale=3/7]
\draw[thick] (0,0)--(2,0)--(2,-1)--(3,-1)--(3,-3)--(6,-3)--(6,-4)--(0,-4)--(0,0);
\draw (0,-3)--(1,-3)--(1,0) (1,-2)--(3,-2) (1,-3)--(2,-3)--(2,-4) (3,-3)--(3,-4);
\node at (.5,-2.6) {1};
\node at (2.5,-3.6) {2};
\node at (2.5,-1.6) {4};
\node at (5.5,-3.6) {3};
\end{tikzpicture}
\caption{A 3-ribbon tableau of shape $\mu=(6,3,3,2)$}
\end{figure}
James and Kerber \cite{MR644144} showed that there is a weight-preserving bijection between semi-standard ribbon tableaux of shape $\mu$, and semi-standard fillings of a $k$-tuple of unshifted shapes $(\mu^0, \mu^1,\ldots,\mu^{k-1})$ called the \emph{$k$-quotient of $\mu$}. This shows that:
\begin{eqnarray}\label{Arib}
GF_{\mu/\mu_0}^{(k)}(X)=s_{\mu^0}(X)s_{\mu^1}(X)\cdots s_{\mu^{k-1}}(X)
\end{eqnarray}
The \emph{spin} of a ribbon $R$, defined by Lascoux, Leclerc and Thibon \cite{MR1434225} is $(|R|-ht(R)-1)/2$, which is not necessarily an integer. For a semi-standard $k$-ribbon tableaux $T$ of shape $\mu$, we define the \emph{spin} of $T$ to be the sum of the spins of all ribbons on $T$. The \emph{cospin} of $T$ is given by $spin(T*)-spin(T)$ where $T*$ is the semi-standard $k$-ribbon tableaux of shape $\mu$ with the maximum spin. The cospin is an integer for every tiling $T$.

Multiplying each tableau by a variable $q$ raised to its cospin gives us the LLT-polynomial: 
\begin{eqnarray}
GF^{(k)}_{\mu/\mu_0}(X;q)=\sum_{T \in SRT_k(\mu)} q^{cospin (T)} F_{Des(T)}(X) 
\end{eqnarray}
The LLT-polynomials can be written as a sum of Schur polynomials with coefficients from $\mathbb{Z}^+[q]$\cite{MR2115257}.
\subsection{Shifted Tableaux}

A partition $\lambda=(\lambda_1, \lambda_2,\ldots,\lambda_k)$ is called \emph{strict} if all its parts are distinct. With every strict partition, we associate a \emph{shifted diagram}, which is an array with $\lambda_i$ boxes on row $i$, where row $i$ is shifted $k-i$ steps to the right, forming a staircase shape. For any cell $C$ on a shifted diagram, we define its diagonal value to be $diag(C)=col(C)-row(C)+1$. Note that the smallest diagonal value is $1$ and is attained only at the leftmost diagonal which is denoted the \emph{main diagonal} of $\lambda$. 
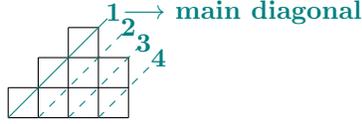
\begin{figure}[ht]
\centering
\begin{tikzpicture}[scale=.4]
\draw (0,0) -- (1,0);
\draw (-1,-1) -- (2,-1);
\draw (-2,-2) -- (2,-2);
\draw (-2,-3) -- (2,-3);
\draw (0,0) -- (0,-3)   (1,0) -- (1,-3);
\draw (-1,-1) -- (-1,-3)  (2,-1) --  (2,-3);
\draw (-2,-2) -- (-2,-3)  ;
\draw[ teal](-2,-3)--(1.3,0.3);
\node[teal, right] at (1.5,0.5) {\textbf{$\longrightarrow$ main diagonal}};
\node[teal] at (1.5,0.5) {$\textbf{1}$};
\draw[dashed, teal](-1,-3)--(1.8,-0.2);
\node[teal] at (2,0) {$\textbf{2}$};
\draw[dashed, teal](0,-3)--(2.3,-0.7);
\node[teal] at (2.5,-0.5) {$\textbf{3}$};
\draw[dashed, teal](1,-3)--(2.8,-1.2);
\node[teal] at (3.0,-1) {$\textbf{4}$};
\end{tikzpicture}
 \caption {The shifted diagram for $\lambda=(4,3,1)$ with diagonals labelled with diagonal values.}
  
\end{figure}

A \emph{semi-standard shifted tableau} of shape $\lambda$ is a filling of its boxes with elements from the marked alphabet $1'<1<2'<2<3'<3<\cdots$ such that each row will be non-decreasing from left to right with no repeated marked numbers, and each column will be non-decreasing from bottom to top with no repeated unmarked numbers. A semi-standard shifted tableau of shape $\lambda$ that contains each of the numbers $1,2,\ldots,|\lambda|$ exactly once (possibly marked), it is called \emph{marked standard}, and if they are all unmarked it is called \emph{standard}. We will denote the set of semi-standard shifted tableaux of shape $\lambda$ by $SSShT(\lambda)$, the set of marked standard ones by $SShT\pm(\lambda)$ and the set of the standard ones by $SShT(\lambda)$. Figure \ref{fig:start} includes the shifted diagram for $\lambda=(4,3,1)$, as well as examples of semi-standard, marked standard and standard shifted tableaux. The tableaux given here are related by a standardization algorithm, which will be introduced in Section \ref{sec:quotients}.

\begin{figure}[ht]
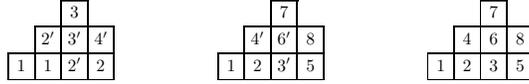

\center
\scalebox{.65}{ \begin{ytableau}
   \none &\none & 3\\
   \none &2' &3' &4' \\
    1& 1&2' &2  \\
  \end{ytableau}\qquad \qquad \qquad  \begin{ytableau}
   \none &\none & 7 \\
   \none &4' &6' &8 \\
    1& 2&3' &5   \\
  \end{ytableau}\qquad \qquad \qquad  \begin{ytableau}
   \none &\none & 7\\
   \none &4 &6 &8 \\
    1& 2&3 &5  \\
  \end{ytableau}}
  \caption {Shifted tableaux examples of shape $\lambda=(4,3,1)$}
  \label{fig:start}
\end{figure}

Schur's Q- and P-functions for a strict partition $\lambda$ are defined as follows:
\begin{eqnarray}
Q_{\lambda}(X)&=&\sum_{S\in SSShT(\lambda)}X^{|S|}\\
P_{\lambda}(X)&=&2^{-ht(\lambda)}\sum_{S\in SSShT(\lambda)}X^{|S|} \quad = \sum_{S\in SSShT^*(\lambda)}X^{|S|}
\end{eqnarray}
where $S\in SSShT^*(\lambda)$ denotes the set of semi-standard tableaux of shape $\lambda$ with no marked entries on the main diagonal, and $X^{|S|}$ is the monomial where the power of $x_i$ is equal to the number of times $i$ or $i'$ occurs in $S$. The semistandard filling in Figure \ref{fig:start}, for example, corresponds to the monomial $x_1^2x_2^3x_3^2x_4$.

The \emph{reading word} of a shifted tableau is, like in the unshifted case,  a reading of all its labels from left to right, top to bottom. These definitions
can be extended to the reading words of marked standard tableaux by
first moving all marked coordinates to the beginning and then
reversing their order and working with the corresponding word, so that Des$(74'6'8123'5)=$Des$(36478125)=\{2,5\}$. 
\begin{lem}\label{lem:7}If $i \in Des(T)$, then $i\in Des(T')$ if and only if $i$ is unmarked in $T'$. If $i \notin Des(T)$, then $i \in Des(T')$ if and only if $i+1$ is marked in $T$.
\end{lem}
\begin{proof} This follows directly from the definition of descents on marked tableaux.
\end{proof}

Like in the case of Schur functions, Schur's Q-functions can be expanded in terms of the fundemental quasisymmetric functions:
$$Q_{\lambda}(X)=\sum_{T'\in SShT\pm(\lambda)}F_{Des(T')}(X)$$
For this expansion, we only look at the marked standard tableaux of shape $\lambda$. An expansion that also eliminates the markings and only considers the standard fillings was given by Stembridge\cite{MR1389788}:
$$Q_{\lambda}(X)=\sum_{T\in SShT(\lambda)}2^{|Peak(T)|+1}G_{Peak(T)}(X)$$
Here, the functions $G_{P}$, where $P$ is a subset of $[2,3,..,n-1]$ with no consecutive entries, are the \emph{peak functions} are defined in \cite{MR1389788} by:
$$ G_P(X)=\sum_{\substack{D\in [n-1]\\ Spike(D) \supset P}}F_{D}(X)$$

Given two strict partitions $\mu$ and $\lambda$ with $\mu \subset \lambda$, the \emph{skew-shifted diagram} $\lambda \backslash \mu$ is the diagram for  $\lambda$ with the cells corresponding to the diagram of $\mu$ deleted.

\begin{figure}[ht]
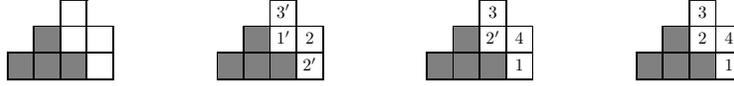

\center
\scalebox{.65}{
     \begin{ytableau}
   \none &\none &  \\
   \none &  *(gray)& & \\
     *(gray) &  *(gray)& *(gray)  & \\
  \end{ytableau}\qquad \qquad \qquad  \begin{ytableau}
   \none &\none & 3'\\
   \none & *(gray) &1' &2 \\
      *(gray) &  *(gray)& *(gray) &2'  \\
  \end{ytableau}\qquad \qquad \qquad  \begin{ytableau}
   \none &\none & 3 \\
   \none & *(gray) &2' &4 \\
     *(gray)&  *(gray)& *(gray) &1   \\
  \end{ytableau}\qquad \qquad \qquad  \begin{ytableau}
   \none &\none & 3\\
   \none & *(gray)&2 &4 \\
    *(gray) &  *(gray)& *(gray) &1  \\
  \end{ytableau}}
  \caption {The skew-shifted diagram $(4,3,1) \backslash (3,1)$ with a corresponding semi-standard, marked standard and standard filling.}
\end{figure}

We can apply the above definitions to the skew-shifted diagrams to get skew-shifted tableaux. More precisely, the set of semi-standard shifted tableaux of shape $\lambda \backslash \mu$, denoted $SSShT(\lambda \backslash \mu)$ is given by all the fillings of $\lambda \backslash \mu$ from the marked alphabet with non-decreasing columns and rows such that we have no unmarked numbers repeated along columns and no marked numbers repeated along rows. We will denote the marked standard fillings of $\lambda \backslash \mu$ (where we use each number from $1$ to $n$ once, possibly marked) by $SShT\pm(\lambda \backslash \mu)$ and its standard fillings  (where we use each number from $1$ to $n$ once, unmarked) by $SShT(\lambda \backslash \mu)$ This gives rise to a skew analogue for Schur's Q-function:
\begin{eqnarray}
Q_{\lambda \backslash \mu}(X)=\sum_{S \in SSShT(\lambda \backslash \mu)} X^{|
S|}=\sum_{T' \in SShT\pm(\lambda \backslash \mu)}F_{Des(T')}(X)=\sum_{T \in SShT(\lambda \backslash \mu)}G_{Peak(T)}(X)
\end{eqnarray}

It was shown by Stembridge that the skew-shifted Q-functions expand positively into Schur's Q-functions:

\begin{thm}\emph{(Stembridge \cite{MR991411})} \label{thm:Stembridge}There exist coefficients $f^{\lambda}_{\mu,\nu} \in \mathbb{N}$ satisfying:
\begin{eqnarray*}
Q_{\lambda \backslash \mu}(X)=\sum_{\nu}f^{\lambda}_{\mu,\nu}Q_{\nu}(X)\qquad Q_{\mu}(X)Q_{\nu}(X)=\sum_{\lambda}f^{\lambda}_{\mu,\nu}Q_{\lambda}(X)
\end{eqnarray*}
where $f^{\lambda}_{\mu,\nu}=0$ unless $|\mu|+|\nu|=|\lambda|$.
\end{thm}
\section{Shifted Ribbon Tableaux}
\subsection{Ribbons on Shifted Diagrams}

On a shifted diagram, we call the columns strictly to the left of the last row its \emph{shifted region}, and the rest its \emph{unshifted region}. Note that the unshifted region uniquely determines the diagram \footnote{We deviate from convention in defining the shifted region, to define hook lengths more easily.}.

The definition of the hook of a cell on a shifted diagram depends on whether the cell falls into the shifted region.
For any cell $C$ in the unshifted region, the \emph{hook} of $C$ is the union of $C$, with the cells above it in its column, and the cells to its right in the row. For a cell in the shifted region, its hook additionally includes the row of cells directly above the highest cell in the column of $C$. The number of cells in its hook is called the \emph{hook length} of $C$.

\begin{figure}[ht]
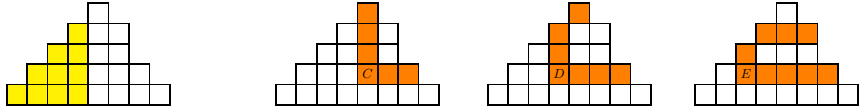

    \centering
    \scalebox{.5}{
\ytableausetup{nosmalltableaux}
\begin{ytableau}
   \none & \none & \none & \none &    \\
   \none &\none &\none & *(yellow) &  & \\
   \none &\none & *(yellow)&*(yellow) & & \\
   \none & *(yellow)& *(yellow)&*(yellow) & & & \\
   *(yellow) & *(yellow)&*(yellow)& *(yellow)& & & &\\
  \end{ytableau}\qquad\qquad\qquad\qquad
  \begin{ytableau}
   \none & \none & \none & \none & *(orange)  \\
   \none &\none &\none &  & *(orange)& \\
   \none &\none & & &*(orange)& \\
   \none & & & &*(orange)C&*(orange)&*(orange)\\
    & & & & & & &\\
  \end{ytableau}  \quad\quad\quad
  \ytableausetup{nosmalltableaux}
  \begin{ytableau}
   \none & \none & \none & \none & *(orange)   \\
   \none &\none &\none &*(orange)  &  & \\
   \none &\none & &*(orange) & & \\
   \none & & & *(orange)D& *(orange)&*(orange) &*(orange) \\
    & & & & & & &\\
  \end{ytableau}\quad\quad\quad
  \ytableausetup{nosmalltableaux}
  \begin{ytableau}
   \none & \none & \none & \none &    \\
   \none &\none &\none &*(orange)  &*(orange)  &*(orange) \\
   \none &\none &*(orange) & & & \\
   \none & & *(orange)E& *(orange)& *(orange)&*(orange) &*(orange) \\
    & & & & & & &\\
  \end{ytableau}}
    \caption{Shifted region of $(8,6,4,3,1)$ and hooks of the cells $C$, $D$ and $E$}
    \label{fig:hooky}
\end{figure}

\begin{defn}We define a \emph{single-ribbon} on  $\lambda$ to be a connected skew-shifted diagram  with each cell on a different diagonal (Equivalently, not containing any $2\times2$ square).
A single ribbon $R$ is \emph{removable} if  $\lambda -R$ is also a shifted diagram.
\end{defn}

Some important notations we will use about ribbons throughout the paper are heads and tails of ribbons. The cell with the highest diagonal value will be called the \emph{head} of $R$, denoted by $H(R)$, and the one with the lowest will be called the \emph{tail} of $R$, denoted $T(R)$. We will also use $|R|$ for the size of a ribbon (the number of cells it contains) and $ht(R)$ for the height of a ribbon (the number of rows of $\lambda$ that it intersects).

Note that for a single-ribbon $R$, $|R|= $diag$(H_R) - $diag$(T_R) +1$

\begin{defn}A \emph{double-ribbon} of size $k$ is a union of two disjoint single-ribbons $R$ and $S$ of sizes $r\geq s$ with $r+s=k$ such that the tail of $R$ is  on the main diagonal of $\lambda$, and the tail of $S$ is on the main diagonal of $\lambda - R$, and their union forms a skew-shifted shape. A double ribbon $Q$ is \emph{removable} if  $\lambda -Q$ is also a shifted diagram.
\end{defn}

The head of $Q=(R,S)$ is the head of $R$, and its tail is the tail of $s$.

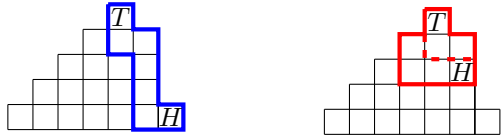
\begin{figure}[ht]
\centering
\begin{tikzpicture}[scale=1/3]

\draw (0,0) -- (1,0);
\draw (-1,-1) -- (2,-1);
\draw (-2,-2) -- (2,-2);
\draw (-3,-3) -- (2,-3);
\draw (-4,-4) -- (3,-4);
\draw (-4,-5) -- (3,-5);

\draw (0,0) -- (0,-5)   (1,0) -- (1,-5);
\draw (-1,-1) -- (-1,-5)  (2,-1) --  (2,-5);
\draw (-2,-2) -- (-2,-5)   (2,-2) -- (2,-5);
\draw (-3,-3) -- (-3,-5)   (2,-3) -- (2,-5);
\draw (-4,-4) -- (-4,-5)   (3,-4)-- (3,-5);

 \draw[ultra thick,blue] (0,0)--(1,0)--(1,-1)--(2,-1)--(2,-4)--(3,-4)--(3,-5)--(1,-5)--(1,-2)--(0,-2)--(0,0);
\node at (0.5,-0.5) {$T$};
\node at (2.5,-4.5) {$H$};
\end{tikzpicture}\qquad \qquad \quad \begin{tikzpicture}[scale=1/3]
\draw (0,0) -- (1,0);
\draw (-1,-1) -- (2,-1);
\draw (-2,-2) -- (2,-2);
\draw (-3,-3) -- (2,-3);
\draw (-4,-4) -- (3,-4);
\draw (-4,-5) -- (3,-5);

\draw (0,0) -- (0,-5)   (1,0) -- (1,-5);
\draw (-1,-1) -- (-1,-5)  (2,-1) --  (2,-5);
\draw (-2,-2) -- (-2,-5)   (2,-2) -- (2,-5);
\draw (-3,-3) -- (-3,-5)   (2,-3) -- (2,-5);
\draw (-4,-4) -- (-4,-5)   (3,-4)-- (3,-5);

 \draw[ultra thick,red] (0,0) -- (1,0)--(1,-1)--(2,-1)--(2,-3)--(2,-3)--(-1,-3)--(-1,-1)--(0,-1)--(0,0) ;
  \draw[ultra thick,dashed, red] (0,-1)--(0,-2)--(2,-2);
\node at (0.5,-0.5) {$T$};
\node at (1.5,-2.5) {$H$};

\end{tikzpicture}
\caption{A single ribbon(left) and a double ribbon (right)}
\end{figure}

\begin{defn} A \emph{$k$-ribbon} is a single or double ribbon of size $k$.
\end{defn}

\begin{prop}For any removable $k$-ribbon $R$ on $\lambda$, there is no cell on $R$ strictly to the right or strictly below $H(R)$.
\label{prop:head}
\end{prop}
\begin{proof} As $\lambda \backslash R$ is also a shifted diagram, if $R$ includes a cell strictly to the right of $H(R)$, it will also contain the cell to the right of $\lambda$ in the same row. Similarly, if $R$ has a cell below $H(R)$, it will also include the cell below $\lambda$ in the same column. In both cases, $R$ has a cell with a higher diagonal value than $H(R)$, giving us a contradiction.
\end{proof}

\begin{prop}A shifted diagram $\lambda$ has a removable single $k$-ribbon with $diag(H_R)=m$ if and only if it has a part of size $m+k$ and no part of size $m$. If it exists, it is the unique ribbon $R$ where $\lambda-R$ has the part $m+k$ replaced with $m$. Furthermore, $\lambda$ has a removable double $k$-ribbon with $diag(H_R)=a$ if and only if it has two parts of size $a$ and $k-a<a$, the ribbon being the unique $R$ where $\lambda-R$ is $\lambda$ with the two parts removed. \label{prop:remove}
\end{prop}
\begin{proof} If $\lambda$ has a part of size $m+k$ and no part of size $m$, then removing the outermost cells from diagonals $m+k$ to $m+1$ gives us a ribbon of size $k$, with $diag(H_R)=m$. As the ribbon is removable,  $\lambda \backslash R$ is itself a skew-shifted diagram, which means there is no cell on $\lambda$ above $T_R$ or to the left of $H_R$, and $R$ contains the outermost cell of each diagonal from $m+k$ to $m+1$.
As $H_R$ is at the end of a row, and has diagonal value $m+k$, $\lambda$ contains a part of size $m+k$. Similarly, $T_R$ is on a row with at least $m+1$ cells, and as the row above has no cell above $T_R$ it has less than $m$ cells. The case for the double ribbon follows as the double ribbon is a union of two single ribbons, and a ribbon of size $a$ with $diag(H_R)=a$ will have its tail on the main diagonal.
\end{proof}

A corollary of this proposition is that no shape can have a double ribbon of size $(t,t)$.

\begin{thm} A shifted diagram $\lambda$ admits no removable $k$-ribbon iff it has no cells with hook length equal to $k$. 
\end{thm}
In this case, we call $\lambda$ a \emph{$k$-core}.

\begin{proof}We claim that there is a bijection between removable $k$-ribbons and cells with hook length equal to $k$, where cells in the unshifted part correspond to single ribbons and cells in the shifted part correspond to double ribbons. Under this bijection it is clear that if a diagram admits no $k$ ribbons, it can not have a cell with hook length $k$ and vice versa. 

Let $C$ have hook length equal to $k$. First let us look at the case when $C$ is in the unshifted part. Let $R$ be the single $k$-ribbon consisting of the outermost cell on each diagonal that the hook of $C$ passes. As the head and tail of $R$ are the endpoints of the hook of $C$, $R$ is removable.  Conversely, if $R$ is a single $k$-ribbon, the cell on the row of $H(R)$ and the column of $T(R)$ has hook length $k$ and is on the unshifted part.

Now let us assume $C$ is a cell in the shifted part, so that its hook includes the row above the column of $C$. Assume this row has length $t$. This means, $C$ is on a row of size $t-k$, and by Proposition \ref{prop:head} the shape has a unique removable double ribbon of size $(t,k-t)$. Conversely, if $R$ is a removable double ribbon of size $(t,k-t)$ with $t<t-k$, the diagram has a rows $i$ and $j$ with sizes $t$ and $t-k$. The cell on  row $j$ and column right below row $i$ falls on the shifted part and has hook length $k$.
\end{proof}

\subsection{The $k$-Abacus Correspondence}

In this section, we will show that the ribbons on shifted tableaux can be expressed using the $k$-abacus notation of James and Kerber\cite{MR644144}.
A $k$-abacus consists of runners labeled by $1, 2, 3, \ldots,k$, and numbers placed on these runners as follows:
\begin{center}
\begin{tabular}{ l l l l l }
  1 & 2 & 3 & $\cdots$ & k\\
  \hline
1 & 2 & 3 & $\cdots$ & k\\
k+1 & k+2 & k+3 & $\cdots$ & 2k\\
2k+1 & 2k+2 & 2k+3 & $\cdots$ & 3k\\
$\cdots$ & $\cdots$ & $\cdots$ & $\cdots$ &$\cdots$ \\
\end{tabular}
\end{center}

Two runners $i$ and $j$ are called \emph{$k$-conjugate} if $i+j=k$. To any shifted diagram $\lambda =(\lambda_1,\lambda_2,\ldots,\lambda_n)$ we will associate the $k$-abacus with beads on positions $\lambda_1,\lambda_2,\ldots,\lambda_n$.
For example, for the diagram $\lambda=(16, 11, 10, 9, 8, 7, 4, 3, 1)$ we get the $5$-abacus:
\begin{center}
\begin{tabular}{ l l l l l }
  $a_1$ & $a_2$ & $a_3$ & $a_4$ & $a_5$\\
  \hline
$\bullet$ & $\circ$ & $\bullet$ & $\bullet$ & $\circ$\\
$\circ$ & $\bullet$ & $\bullet$ & $\bullet$  & $\bullet$\\
$\bullet$  & $\circ$ & $\circ$ & $\circ$ & $\circ$\\
$\bullet$ & $\circ$ & $\circ$ & $\circ$ & $\circ$\\
$\circ$ & $\circ$ & $\circ$ & $\circ$ & $\circ$\\
$\cdots$  & $\cdots$  &$\cdots$  & $\cdots$ &$\cdots$ \\
\end{tabular}
\end{center}

Given a strict partition $\lambda$, we identify each runner $a_i$ in its abacus with a shifted shape $\alpha^{(i)}$, by treating the runners as the abacus of a shifted $1$-core. More precisely, $\alpha^{(i)}$ will be the shifted shape $(\alpha^{(i)}_1,\alpha^{(i)}_2,\ldots,\alpha^{(i)}_t)$ where $k(\alpha^{(i)}_1-1)+i,\ldots,k(\alpha^{(i)}_t-1)+i$ are the parts of $\lambda$ that are equal to  $i$ modulo $k$. We will call the $k$-tuple $(\alpha^{(1)},\alpha^{(2)},\ldots,\alpha^{(n)})$  the abacus representation $\lambda$. The $5$-abacus representation of the example $(16, 11, 10, 9, 8, 7, 4, 3, 1)$ with the abacus above is $(431,2,21,21,2)$.

There are two types of moves allowed on the $k$-abacus of $\lambda$\cite{MR1264418}:
\begin{itemize}
\item  \emph{Type I Move:} Sliding one bead one position higher in its runner if that position is unoccupied, or removing a bead from the top row of column $k$
\item \emph{Type II Move:} Removing two beads from the first row, if they are on two conjugate runners.
\end{itemize}

After a move on the $k$-abacus, we get a new shifted diagram $\lambda* \subset \lambda$. A \emph{Type I} move corresponds to replacing a part of size $m+k$ with one of size $m$, whereas a \emph{Type II} move corresponds to removing two parts of sizes adding up to $k$. By Proposition \ref{prop:remove}, we have the following correspondence:

\begin{thm}\label{cor:ribboncorrespondence} Making a move on the $k$-abacus of $\lambda$ is equivalent to removing a $k$-ribbon from $\lambda$. In particular,
\begin{enumerate}
\item \emph{Single-Ribbon Correspondence:} Making the Type I move from position $m+k$ to position $m$ is equivalent to removing a single-ribbon A with diag$(H_A)=m+k$ and diag$(T_A)=m+1$ (where the case $m=0$ is removing a bead from the top row of column $k$) .
\item \emph{Double-Ribbon Correspondence:} Making the Type II move removing top beads $t$ and $k-t$ from conjugate runners $a_t$ and $a_{k-t}$ equivalent to removing a double-ribbon of size $(t,k-t)$.
\end{enumerate}
\end{thm}
This theorem implies that the $k$-core of a shifted diagram corresponds to the $k$-core of the abacus. As the final state of the abacus is independent of the order of the moves \cite{MR644144}, we have the following corollary.
\begin{cor}The $k$-core of a shifted diagram is unique.
\end{cor}

\subsection{Shifted Ribbon Tableaux}

\begin{defn} A \emph{standard k-ribbon tableau} of shape $\lambda$ is a sequence of shifted diagrams $\lambda^{(0)} \subset \lambda^{(1)} \subset \cdots  \subset \lambda^{(n)}=\lambda$, where each $A_i=\lambda^{(i)} \backslash \lambda^{(i-1)}$, $i=1,2..n$ is a k-ribbon, and $\lambda^{(0)}$ is a k-core. For each $i$, we label ribbon $A_i$ with	 $i$. 
\end{defn}

\begin{figure}[ht]
\centering
\begin{tikzpicture}[scale=1/3]
\draw[thick] (0,0)--(1,0)--(1,-1)--(2,-1)--(2,-4)--(1,-4)--(1,-2)--(0,-2)--(0,0) ;

\draw[thick] (0,-2)--(0,-5)--(3,-5)--(3,-4)--(2,-4);

\draw[thick] (-2,-3)--(-3,-3)--(-3,-4)--(-4,-4)--(-4,-5)--(0,-5);

\draw[thick] (0,-1)--(-1,-1)--(-1,-2)--(-2,-2)--(-2,-3)--(-2,-4)--(0,-4);
\node at (1,-1.5) {4};
\node    at (0.5,-4.5) {3};
\node   at (-1,-3) {2};
\node  at (-2.5,-4.5) {1};
\end{tikzpicture}
\caption{A 5-Ribbon Tableau with no 5-core of shape $\lambda=(7,5,4,3,1)$ }
\end{figure}
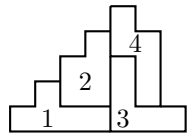

\begin{defn} A skew-shifted diagram $S$ on a shifted diagram $\lambda$ is called a \emph{horizontal $k$-ribbon strip} (resp. \emph{vertical $k$-ribbon strip}) if there exists a sequence of shifted diagrams  $\lambda^{(0)} \subset \lambda^{(1)} \subset \cdots  \subset \lambda^{(t)}=\lambda$, where:
\begin{enumerate}
    \item Each $R_i:=\lambda^{(i)}\setminus\lambda^{(i-1)}$ is a $k$-ribbon.
    \item $H(R_i)$ is strictly to the right of (resp. strictly above) $H(R_{i-1})$ for each $i$.
    \item $S= \bigcup_{i=1}^n R_{i} = \lambda \setminus \lambda^{(0)}$.  
\end{enumerate}
\end{defn}

\begin{figure}[H]
    \begin{tikzpicture}[scale=3/7]
    \draw (0,1)--(3,1)--(3,2)--(0,2)--(0,1);
    \draw (3,2)--(4,2)--(4,1)--(5,1)--(5,0)--(3,0)--(3,1);
    \node at (2.5,1.5) {H};
    \node at (4.5,0.5) {H};
    \node[teal] at (3,-0.5) {\Checkmark};
    \end{tikzpicture}\qquad
        \begin{tikzpicture}[scale=3/7]
    \draw (0,0)--(3,0)--(3,1)--(0,1)--(0,0);
    \draw (3,1)--(3,2)--(4,2)--(4,1)--(5,1)--(5,0)--(3,0);
    \node at (2.5,0.5) {H};
    \node at (4.5,0.5) {H};
    \node[teal] at (3,-0.5) {\Checkmark};
    \end{tikzpicture}\qquad
            \begin{tikzpicture}[scale=3/7]
    \draw (0,0)--(3,0)--(3,1)--(0,1)--(0,0);
    \draw (2,1)--(2,2)--(4,2)--(4,0)--(3,0);
    \node at (2.5,0.5) {H};
    \node at (3.5,0.5) {H};
    \node[teal] at (2.5,-0.5) {\Checkmark};
    \end{tikzpicture}\qquad \quad \qquad
        \begin{tikzpicture}[scale=3/7]
    \draw (0,0)--(3,0)--(3,1)--(0,1)--(0,0);
    \draw (0,1)--(0,2)--(3,2)--(3,1);
    \node at (2.5,0.5) {H};
    \node at (2.5,1.5) {H};
    \node[red] at (2,-0.5) {\XSolidBrush};
    \end{tikzpicture}\qquad
        \begin{tikzpicture}[scale=3/7]
    \draw (0,0)--(3,0)--(3,1)--(0,1)--(0,0);
    \draw (0,1)--(-1,1)--(-1,2)--(2,2)--(2,1);
    \node at (2.5,0.5) {H};
    \node at (1.5,1.5) {H};
    \node[red] at (2,-0.5) {\XSolidBrush};
    \end{tikzpicture}\qquad
    \begin{tikzpicture}[scale=3/7]
    \draw (0,0)--(3,0)--(3,1)--(0,1)--(0,0);
    \draw (0,1)--(0,3)--(1,3)--(1,2)--(2,2)--(2,1);
    \node at (2.5,0.5) {H};
    \node at (1.5,1.5) {H};
    \node[red] at (1,-0.5) {\XSolidBrush};
    \end{tikzpicture}
    \caption{Horizontal strip examples and non-examples}
\end{figure}
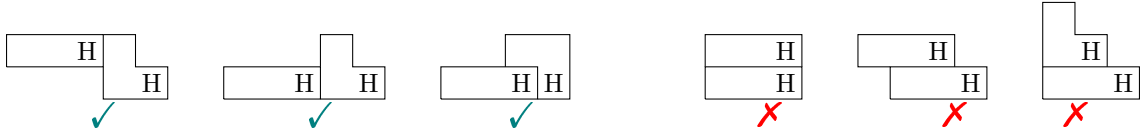

\begin{defn} A \emph{semi-standard shifted $k$-ribbon tableau} is given by a sequence $\lambda^{(0)} \subset \lambda^{(1')} \subset \lambda^{(1)} \subset \lambda^{(2')}\subset \lambda^{(2)} \subset \cdots  \subset \lambda^{(n)}=\lambda$, where:
\begin{itemize}
    \item $\lambda^{(0)}$ is the $k$-core of $\lambda$.
    \item Each  $\lambda^{(i)} \setminus \lambda^{(i')}$ is a (possibly empty) horizontal $k$-strip. 
    \item Each  $\lambda^{(i')} \setminus\lambda^{(i-1)}$ is a (possibly empty) vertical $k$-strip.
\end{itemize} We number the ribbons on the strip $\lambda^{(i)} \setminus\lambda^{(i')}$ by $i$ and the ribbons forming the strip $\lambda^{(i')} \setminus\lambda^{(i-1)}$ by $i'$ for each $i=1,2,\ldots,n$. 
\end{defn}

\begin{figure}[H]
\scalebox{.7}{
    \begin{tikzpicture}[scale=1/2]
\draw (0,0)--(1,0)--(1,-2)--(2,-2)--(2,-4)--(-3,-4)--(-3,-3)--(-2,-3)--(-2,-2)--(-1,-2)--(-1,-1)--(0,-1)--(0,0);
\draw (1,-2)--(-1,-2)--(-1,-4) (-1,-3)--(2,-3);
\node  at (0.5,-1.5) {$2'$};
\node  at (1.5,-2.5) {$2'$};
\node  at (1.5,-3.5) {1};
\node  at (-1.5,-3.5) {$1'$};
\end{tikzpicture} \quad
 \begin{tikzpicture}[scale=1/2]
\draw (0,0)--(1,0)--(1,-2)--(2,-2)--(2,-4)--(-3,-4)--(-3,-3)--(-2,-3)--(-2,-2)--(-1,-2)--(-1,-1)--(0,-1)--(0,0);
\draw (1,-2)--(-1,-2)--(-1,-4) (-1,-3)--(2,-3);
\node  at (0.5,-1.5) {2};
\node  at (1.5,-2.5) {$2'$};
\node  at (1.5,-3.5) {1};
\node  at (-1.5,-3.5) {$1'$};
\end{tikzpicture}\quad
\begin{tikzpicture}[scale=1/2]
\draw (0,0)--(1,0)--(1,-2)--(2,-2)--(2,-4)--(-3,-4)--(-3,-3)--(-2,-3)--(-2,-2)--(-1,-2)--(-1,-1)--(0,-1)--(0,0);
\draw (1,-2)--(-1,-2)--(-1,-4) (-1,-3)--(2,-3);
\node  at (0.5,-1.5) {$2'$};
\node  at (1.5,-2.5) {$2'$};
\node  at (1.5,-3.5) {1};
\node  at (-1.5,-3.5) {1};
\end{tikzpicture}\quad
 \begin{tikzpicture}[scale=1/2]
\draw (0,0)--(1,0)--(1,-2)--(2,-2)--(2,-4)--(-3,-4)--(-3,-3)--(-2,-3)--(-2,-2)--(-1,-2)--(-1,-1)--(0,-1)--(0,0);
\draw (1,-2)--(-1,-2)--(-1,-4) (-1,-3)--(2,-3);
\node  at (0.5,-1.5) {2};
\node  at (1.5,-2.5) {$2'$};
\node  at (1.5,-3.5) {1};
\node  at (-1.5,-3.5) {1};
\end{tikzpicture} \quad
\begin{tikzpicture}[scale=1/2]
\draw (0,0)--(1,0)--(1,-2)--(2,-2)--(2,-4)--(-3,-4)--(-3,-3)--(-2,-3)--(-2,-2)--(-1,-2)--(-1,-1)--(0,-1)--(0,0);
\draw (1,-2)--(-1,-2)--(-1,-4) (-1,-3)--(2,-3);
\node  at (0.5,-1.5) {$2'$};
\node  at (1.5,-2.5) {$2'$};
\node  at (1.5,-3.5) {$2'$};
\node  at (-1.5,-3.5) {$1'$};
\end{tikzpicture}\quad 
 \begin{tikzpicture}[scale=1/2]
\draw (0,0)--(1,0)--(1,-2)--(2,-2)--(2,-4)--(-3,-4)--(-3,-3)--(-2,-3)--(-2,-2)--(-1,-2)--(-1,-1)--(0,-1)--(0,0);
\draw (1,-2)--(-1,-2)--(-1,-4) (-1,-3)--(2,-3);
\node  at (0.5,-1.5) {2};
\node  at (1.5,-2.5) {$2'$};
\node  at (1.5,-3.5) {$2'$};
\node  at (-1.5,-3.5) {$1'$};
\end{tikzpicture} \quad
\begin{tikzpicture}[scale=1/2]
\draw (0,0)--(1,0)--(1,-2)--(2,-2)--(2,-4)--(-3,-4)--(-3,-3)--(-2,-3)--(-2,-2)--(-1,-2)--(-1,-1)--(0,-1)--(0,0);
\draw (1,-2)--(-1,-2)--(-1,-4) (-1,-3)--(2,-3);
\node  at (0.5,-1.5) {$2'$};
\node  at (1.5,-2.5) {$2'$};
\node  at (1.5,-3.5) {$2'$};
\node  at (-1.5,-3.5) {1};
\end{tikzpicture}\quad 
 \begin{tikzpicture}[scale=1/2]
\draw (0,0)--(1,0)--(1,-2)--(2,-2)--(2,-4)--(-3,-4)--(-3,-3)--(-2,-3)--(-2,-2)--(-1,-2)--(-1,-1)--(0,-1)--(0,0);
\draw (1,-2)--(-1,-2)--(-1,-4) (-1,-3)--(2,-3);
\node  at (0.5,-1.5) {2};
\node  at (1.5,-2.5) {$2'$};
\node  at (1.5,-3.5) {$2'$};
\node  at (-1.5,-3.5) {1};
\end{tikzpicture}}

\vspace{.5cm}
\scalebox{.7}{
\begin{tikzpicture}[scale=1/2]
\draw (0,0)--(1,0)--(1,-2)--(2,-2)--(2,-4)--(-3,-4)--(-3,-3)--(-2,-3)--(-2,-2)--(-1,-2)--(-1,-1)--(0,-1)--(0,0);
\draw (1,-2)--(-1,-2)--(-1,-4) (0,-2)--(0,-3)--(1,-3)--(1,-4);
\node at (0.5,-1.5) {$2'$};
\node at (0.5,-3.5) {$1$};
\node at  (1.5,-3.5) {$1$};
\node at (-1.5,-3.5) {$1'$};
\end{tikzpicture}\quad 
\begin{tikzpicture}[scale=1/2]
\draw (0,0)--(1,0)--(1,-2)--(2,-2)--(2,-4)--(-3,-4)--(-3,-3)--(-2,-3)--(-2,-2)--(-1,-2)--(-1,-1)--(0,-1)--(0,0);
\draw (1,-2)--(-1,-2)--(-1,-4) (0,-2)--(0,-3)--(1,-3)--(1,-4);
\node at (0.5,-1.5) {$2$};
\node at (0.5,-3.5) {$1$};
\node at  (1.5,-3.5) {$1$};
\node at (-1.5,-3.5) {$1'$};
\end{tikzpicture}\quad 
 \begin{tikzpicture}[scale=1/2]
\draw (0,0)--(1,0)--(1,-2)--(2,-2)--(2,-4)--(-3,-4)--(-3,-3)--(-2,-3)--(-2,-2)--(-1,-2)--(-1,-1)--(0,-1)--(0,0);
\draw (1,-2)--(-1,-2)--(-1,-4) (0,-2)--(0,-3)--(1,-3)--(1,-4);
\node at (0.5,-1.5) {$2'$};
\node at (0.5,-3.5) {$1$};
\node at  (1.5,-3.5) {$1$};
\node at (-1.5,-3.5) {$1$};
\end{tikzpicture}\quad  
 \begin{tikzpicture}[scale=1/2]
\draw (0,0)--(1,0)--(1,-2)--(2,-2)--(2,-4)--(-3,-4)--(-3,-3)--(-2,-3)--(-2,-2)--(-1,-2)--(-1,-1)--(0,-1)--(0,0);
\draw (1,-2)--(-1,-2)--(-1,-4) (0,-2)--(0,-3)--(1,-3)--(1,-4);
\node at (0.5,-1.5) {2};
\node at (0.5,-3.5) {$1$};
\node at  (1.5,-3.5) {$1$};
\node at (-1.5,-3.5) {1};
\end{tikzpicture}\quad
\begin{tikzpicture}[scale=1/2]
\draw (0,0)--(1,0)--(1,-2)--(2,-2)--(2,-4)--(-3,-4)--(-3,-3)--(-2,-3)--(-2,-2)--(-1,-2)--(-1,-1)--(0,-1)--(0,0);
\draw (1,-2)--(-1,-2)--(-1,-4) (0,-2)--(0,-3)--(1,-3)--(1,-4);
\node at (0.5,-1.5) {$2'$};
\node at (0.5,-3.5) {$1$};
\node at  (1.5,-3.5) {$2'$};
\node at (-1.5,-3.5) {$1'$};
\end{tikzpicture}\quad 
 \begin{tikzpicture}[scale=1/2]
\draw (0,0)--(1,0)--(1,-2)--(2,-2)--(2,-4)--(-3,-4)--(-3,-3)--(-2,-3)--(-2,-2)--(-1,-2)--(-1,-1)--(0,-1)--(0,0);
\draw (1,-2)--(-1,-2)--(-1,-4) (0,-2)--(0,-3)--(1,-3)--(1,-4);
\node at (0.5,-1.5) {2};
\node at (0.5,-3.5) {$1$};
\node at  (1.5,-3.5) {$2'$};
\node at (-1.5,-3.5) {$1'$};
\end{tikzpicture} \quad 
\begin{tikzpicture}[scale=1/2]
\draw (0,0)--(1,0)--(1,-2)--(2,-2)--(2,-4)--(-3,-4)--(-3,-3)--(-2,-3)--(-2,-2)--(-1,-2)--(-1,-1)--(0,-1)--(0,0);
\draw (1,-2)--(-1,-2)--(-1,-4) (0,-2)--(0,-3)--(1,-3)--(1,-4);
\node at (0.5,-1.5) {$2'$};
\node at (0.5,-3.5) {$1$};
\node at  (1.5,-3.5) {$2'$};
\node at (-1.5,-3.5) {1};
\end{tikzpicture}\quad 
 \begin{tikzpicture}[scale=1/2]
\draw (0,0)--(1,0)--(1,-2)--(2,-2)--(2,-4)--(-3,-4)--(-3,-3)--(-2,-3)--(-2,-2)--(-1,-2)--(-1,-1)--(0,-1)--(0,0);
\draw (1,-2)--(-1,-2)--(-1,-4) (0,-2)--(0,-3)--(1,-3)--(1,-4);
\node at (0.5,-1.5) {2};
\node at (0.5,-3.5) {$1$};
\node at  (1.5,-3.5) {$2'$};
\node at (-1.5,-3.5) {1};
\end{tikzpicture}}

\vspace{.5cm}

\caption{Semi-standard $5$-ribbon fillings of $(5,4,2,1)$ with numbers $\leq 2$}
\label{fig:3fillsof5421}
\end{figure}
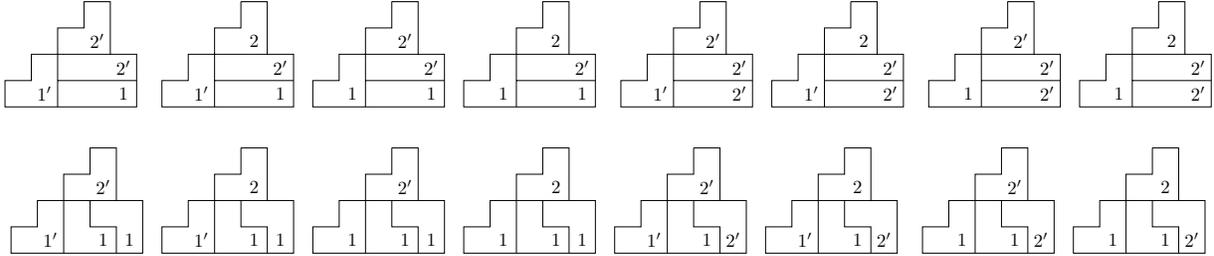

\begin{defn}  For a shifted shape $\lambda$, we define its $k$-ribbon Q and P-functions as follows:
\begin{eqnarray*}
RQ^{(k)}_{\lambda}(X)=\sum_{S\in SSShT^{(k)}(\lambda)}X^{|S|}, \qquad \qquad
RP^{(k)}_{\lambda}(X)=\sum_{S\in SSShT^{*(k)}(\lambda)}X^{|S|}
\end{eqnarray*}
where $SSShT^{(k)}(\lambda)$ is the set of semi-standard shifted $k$-ribbon tableaux of shape $\lambda$, and $SSShT^{*(k)}(\lambda)$ is its subset consisting of tableaux with no marked entries on the ribbons that have boxes on the main diagonal.
\end{defn}

The example illustrated in Figure \ref{fig:3fillsof5421} gives the following ribbon Q- and P-functions restricted to two variables:
\begin{eqnarray*}
RQ^{(3)}_{(5,4,2,1)}(x_1,x_2)&=&4x_1^3x_2+8x_1^2x_2^2+4x_1x_2^3=Q_{3,1}(x_1,x_2)
\\
RP^{(3)}_{(5,4,2,1)}(x_1,x_2)&=&x_1^3x_2+2x_1^2x_2^2+x_1x_2^3=P_{3,1}(x_1,x_2)
\end{eqnarray*}
Note that in this example, the ribbon $Q$ and $P$ functions are Schur Q and P positive respectively. One of our main results in this paper will be to show that the positivity holds  true in general, and give a formula for the ribbon Q-functions in terms of Schur's Q-functions.

\begin{prop}All the shifted $k$-ribbon tableux of shape $\lambda$ have the same number of ribbons that have a box on the main diagonal, called the \emph{$k$-length} of $\lambda$, denoted $\ell^{(k)}(\lambda)$. Consequentally, the Q and P $k$-ribbon functions for $\lambda$ are related by a scalar:
$$RQ^{(k)}_{\lambda}(X)=2^{\ell^{(k)}(\lambda)} RP^{(k)}_{\lambda}(X)$$
\end{prop}
\begin{proof} By Theorem \ref{cor:ribboncorrespondence}, the ribbons that have a box on the main diagonal correspond to the moves on the abacus where beads are removed. The total number is independent of the order of the moves. In fact,  if we denote the number of beads on runner $i$ by $|a_i|$ then:
$$\ell^{(k)}(\lambda)=|a_k|+\sum_{i<(k/2)}max\{|a_i|,|a_{k-i}|\}$$
\end{proof}

\section{Folded Tableaux}

In this section, we will introduce an operation combining two shifted shapes to get an unshifted shape with a specialized diagonal, which we will call a folded diagram. We will later use the folded diagrams, along with their corresponding tableaux in describing the $k$-quotient of a shifted shape. We will use the notation $\delta_n$ to denote the staircase partition $(n,n-1,\ldots,1)$ of size $n$.

\begin{defn}A \emph{folded diagram} $\Gamma=(\gamma,\mathfrak{d})$ is an unshifted diagram $\gamma$ called the \emph{underlying shape} of $\Gamma$ along with a specialized main diagonal $\mathfrak{d}$ which does not necessarily intersect $\gamma$.
\end{defn}

\begin{defn}Let us say we have a pair of shifted shapes (or equivalently strict partitions) $\alpha $ and $\beta$. Denote $n=$min$\{ht(\alpha),ht(\beta)\}$. Their \emph{combination}, which will be denoted by $\alpha \diamond \beta$ will be the folded diagram we obtain by:
\begin{itemize}
\item Step I: If one of the shapes has height $m$ larger than $n$, delete its $m-n$ leftmost columns, so that both shapes have the same number of boxes in their main diagonals. 
\item Step II: Transpose $\alpha$.
\item  Step III: Paste the two diagrams together along their main diagonals, and label this diagonal $\mathfrak{d}$.
\end{itemize} 
 \end{defn}

\textbf{Example:}
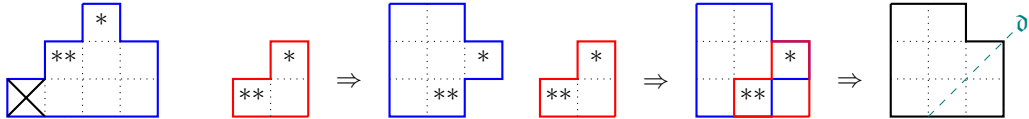
\begin{figure}[H]
    \centering
  \begin{tikzpicture}[scale=1/2]
\draw[dotted ] (0,-1) -- (1,-1);
\draw[dotted ] (-1,-2) -- (2,-2);
\draw[dotted ] (-1,-3) -- (2,-3);
\draw[blue, thick] (0,0)--(0,-1)--(-1,-1)--(-1,-2)--(-2,-2)--(-2,-3)--(2,-3)--(2,-1)--(1,-1)--(1,0)--(0,0);
\draw[red, thick]
(6,-3)--(4,-3)--(4,-2)--(5,-2)--(5,-1)--(6,-1)--(6,-3);
\draw[dotted ] (0,-1) -- (0,-3)   (1,-1) -- (1,-3) ;
\draw [dotted ] (-1,-2)--(-1,-3) ;
\draw[thick] (-2,-3)--(-1,-2) (-2,-2)--(-1,-3);
\draw[dotted ]  (5,-3)--(5,-2)--(6,-2);
\node  at (5.5,-1.5) {*};
\node  at (4.5,-2.5) {**};
\node  at (0.5,-0.5) {*};
\node  at (-0.5,-1.5) {**};
\end{tikzpicture}\quad \raisebox{.4cm}{$\Rightarrow$}\quad
  \begin{tikzpicture}[scale=1/2]
\draw[dotted ] (0,-1) -- (2,-1);
\draw[dotted ] (0,-2) -- (2,-2);
\draw[blue, thick] (0,0)--(2,0)--(2,-1)--(3,-1)--(3,-2)--(2,-2)--(2,-3)--(0,-3)--(0,0);
\draw[red, thick]
(6,-3)--(4,-3)--(4,-2)--(5,-2)--(5,-1)--(6,-1)--(6,-3);
\draw[dotted ] (1,0) -- (1,-3) ;
\draw [dotted ] (2,-1)--(2,-2) ;
\draw[dotted ]  (5,-3)--(5,-2)--(6,-2);
\node  at (5.5,-1.5) {*};
\node  at (4.5,-2.5) {**};
\node  at (2.5,-1.5) {*};
\node  at (1.5,-2.5) {**};
\end{tikzpicture}\quad \raisebox{.4cm}{$\Rightarrow$}\quad
\begin{tikzpicture}[scale=1/2]
\draw[dotted ] (0,-2)--(1,-2);
\draw[dotted ] (0,-1) -- (2,-1);
\draw[dotted ] (1,0)--(1,-2) ;
\draw[blue, thick] (0,0)--(2,0)--(2,-1)--(3,-1)--(3,-2)--(2,-2)--(2,-3)--(0,-3)--(0,0);
\draw[red, thick] (3,-3)--(1,-3)--(1,-2)--(2,-2)--(2,-1)--(3,-1)--(3,-3);
\node   at (1.5,-2.5) {**};
\node   at (2.5, -1.5) {*};
\end{tikzpicture}\quad \raisebox{.4cm}{$\Rightarrow$}\quad
\begin{tikzpicture}[scale=1/2]
\draw[dotted ] (0,-2)--(3,-2);
\draw[dotted ] (0,-1) -- (2,-1);
\draw[dotted ] (1,0)--(1,-3) (2,0)--(2,-3) (3,-1)--(3,-3);
\draw[thick] (0,0)--(2,0)--(2,-1)--(3,-1)--(3,-3)--(0,-3)--(0,0);
\draw[dashed, teal](1,-3)--(3.3,-0.7);
\node[teal] at (3.5,-0.5) {$\mathfrak{d}$};

\end{tikzpicture}
    \caption{The combination of $\alpha=(4,3,1)$ and $\beta=(2,1)$}
\end{figure}

\begin{prop}Any given folded diagram $(\gamma, \mathfrak{d})$ can be uniquely described as a combination of two shifted shapes.
\end{prop}
\begin{proof}
Let us denote the difference $\lambda\backslash\delta_{height(\lambda)}$ by $\lambda'$. If $\mathfrak{d}$ is on the shape, going through $t\geq 0$ boxes, then $(\gamma,\mathfrak{d})=\alpha \diamond \beta$ where $\alpha = \delta_{k+t} +(\{\gamma_{t+1},\gamma_{t+2},\ldots,\gamma{n}\})'$ and   $\beta=\delta_t+\{\gamma_{k+t+1}',\gamma_{k+t+2}',\ldots,\gamma_{n'}'\}$.

If $\mathfrak{d}$ is $k\geq 0$ units to the right of the bottom left corner and outside the shape (the case $k<0$ is symmetrical) we have $(\gamma,\mathfrak{d})=\alpha \diamond \beta$ where $\alpha = \delta_k +(\gamma)'$ and $\beta$ is empty.
\end{proof}

\begin{defn}A \emph{standard folded tableau} of shape $\Gamma=(\gamma,\mathfrak{d})$ is a filling of the cells of $\gamma$ using numbers $1,2,\ldots,n$ each exactly once, with numbers increasing left to right and bottom to top.
\end{defn}

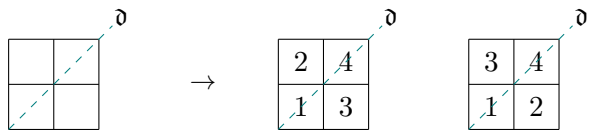
\begin{figure}[H]
    \centering
\begin{tikzpicture}[scale=0.6]
\draw (0,0) -- (2,0) (0,1)--(2,1) (0,2)--(2,2);
\draw (0,0) -- (0,2)   (1,0) -- (1,2) (2,0)--(2,2);
\draw[dashed,teal] (0,0) -- (2.3,2.3);
\node at (2.5,2.5) {$\mathfrak{d}$};
\end{tikzpicture}\qquad \raisebox{0.55cm}{$\rightarrow$} \qquad
\begin{tikzpicture}[scale=0.6]
\draw (0,0) -- (2,0) (0,1)--(2,1) (0,2)--(2,2);
\draw (0,0) -- (0,2)   (1,0) -- (1,2) (2,0)--(2,2);
\node[font=\large] at (0.5,0.5) {1};
\node[font=\large] at (0.5,1.5) {2};
\node[font=\large] at (1.5,0.5) {3};
\node[font=\large] at (1.5,1.5) {4};
\draw[dashed,teal] (0,0) -- (2.3,2.3);
\node at (2.5,2.5) {$\mathfrak{d}$};
\end{tikzpicture} \qquad
\begin{tikzpicture}[scale=0.6]
\draw (0,0) -- (2,0) (0,1)--(2,1) (0,2)--(2,2);
\draw (0,0) -- (0,2)   (1,0) -- (1,2) (2,0)--(2,2);
\node[font=\large] at (0.5,0.5) {1};
\node[font=\large] at (0.5,1.5) {3};
\node[font=\large] at (1.5,0.5) {2};
\node[font=\large] at (1.5,1.5) {4};
\draw[dashed,teal] (0,0) -- (2.3,2.3);
\node at (2.5,2.5) {$\mathfrak{d}$};
\end{tikzpicture} 
 \label{fig:foldst}
    \caption{Standard Folded Tableaux of shape $\Gamma=((2,2),\mathfrak{d})$}
\end{figure}

\begin{defn} A \emph{semi-standard folded tableau} of shape $\Gamma=(\gamma,\mathfrak{d})$ is a semi-standard filling of the skew-shifted diagram $\gamma$ with the rules inverted weakly above the specialized diagonal (above $\mathfrak{d}$ we can have no repeated unmarked numbers on the same row, and no two repeated marked numbers on the same column).
\end{defn}

\begin{figure}[H]
    \centering
\begin{tikzpicture}[scale=0.6]
\draw (0,0) -- (2,0) (0,1)--(2,1) (0,2)--(2,2);
\draw (0,0) -- (0,2)   (1,0) -- (1,2) (2,0)--(2,2);
\draw[dashed,teal] (0,0) -- (2.3,2.3);
\node at (2.5,2.5) {$\mathfrak{d}$};
\end{tikzpicture}\raisebox{0.55cm}{$\rightarrow$} \enskip
\begin{tikzpicture}[scale=0.6]
\draw (0,0) -- (2,0) (0,1)--(2,1) (0,2)--(2,2);
\draw (0,0) -- (0,2)   (1,0) -- (1,2) (2,0)--(2,2);
\node[font=\large] at (0.5,0.5) {1};
\node[font=\large] at (0.5,1.5) {1};
\node[font=\large] at (1.5,0.5) {1};
\node[font=\large] at (1.5,1.5) {$2'$};
\end{tikzpicture} \enskip
\begin{tikzpicture}[scale=0.6]
\draw (0,0) -- (2,0) (0,1)--(2,1) (0,2)--(2,2);
\draw (0,0) -- (0,2)   (1,0) -- (1,2) (2,0)--(2,2);
\node[font=\large] at (0.5,0.5) {1};
\node[font=\large] at (0.5,1.5) {$2'$};
\node[font=\large] at (1.5,0.5) {1};
\node[font=\large] at (1.5,1.5) {$2'$};
\end{tikzpicture} \enskip\begin{tikzpicture}[scale=0.6]
\draw (0,0) -- (2,0) (0,1)--(2,1) (0,2)--(2,2);
\draw (0,0) -- (0,2)   (1,0) -- (1,2) (2,0)--(2,2);
\node[font=\large] at (0.5,0.5) {1};
\node[font=\large] at (0.5,1.5) {1};
\node[font=\large] at (1.5,0.5) {$2'$};
\node[font=\large] at (1.5,1.5) {$2'$};
\end{tikzpicture} \enskip\begin{tikzpicture}[scale=0.6]
\draw (0,0) -- (2,0) (0,1)--(2,1) (0,2)--(2,2);
\draw (0,0) -- (0,2)   (1,0) -- (1,2) (2,0)--(2,2);
\node[font=\large] at (0.5,0.5) {1};
\node[font=\large] at (0.5,1.5) {$2'$};
\node[font=\large] at (1.5,0.5) {$2'$};
\node[font=\large] at (1.5,1.5) {$2'$};
\end{tikzpicture} 
 \enskip
\begin{tikzpicture}[scale=0.6]
\draw (0,0) -- (2,0) (0,1)--(2,1) (0,2)--(2,2);
\draw (0,0) -- (0,2)   (1,0) -- (1,2) (2,0)--(2,2);
\node[font=\large] at (0.5,0.5) {$1'$};
\node[font=\large] at (0.5,1.5) {1};
\node[font=\large] at (1.5,0.5) {1};
\node[font=\large] at (1.5,1.5) {$2'$};
\end{tikzpicture} \enskip
\begin{tikzpicture}[scale=0.6]
\draw (0,0) -- (2,0) (0,1)--(2,1) (0,2)--(2,2);
\draw (0,0) -- (0,2)   (1,0) -- (1,2) (2,0)--(2,2);
\node[font=\large] at (0.5,0.5) {$1'$};
\node[font=\large] at (0.5,1.5) {$2'$};
\node[font=\large] at (1.5,0.5) {1};
\node[font=\large] at (1.5,1.5) {$2'$};
\end{tikzpicture} \enskip\begin{tikzpicture}[scale=0.6]
\draw (0,0) -- (2,0) (0,1)--(2,1) (0,2)--(2,2);
\draw (0,0) -- (0,2)   (1,0) -- (1,2) (2,0)--(2,2);
\node[font=\large] at (0.5,0.5) {$1'$};
\node[font=\large] at (0.5,1.5) {1};
\node[font=\large] at (1.5,0.5) {$2'$};
\node[font=\large] at (1.5,1.5) {$2'$};
\end{tikzpicture} \enskip\begin{tikzpicture}[scale=0.6]
\draw (0,0) -- (2,0) (0,1)--(2,1) (0,2)--(2,2);
\draw (0,0) -- (0,2)   (1,0) -- (1,2) (2,0)--(2,2);
\node[font=\large] at (0.5,0.5) {$1'$};
\node[font=\large] at (0.5,1.5) {$2'$};
\node[font=\large] at (1.5,0.5) {$2'$};
\node[font=\large] at (1.5,1.5) {$2'$};
\end{tikzpicture} 
    \caption{Semi-standard folded tableaux of shape $\Gamma=((2,2),\mathfrak{d})$, using numbers $\leq 2'$}
\label{fig:foldss}
\end{figure}
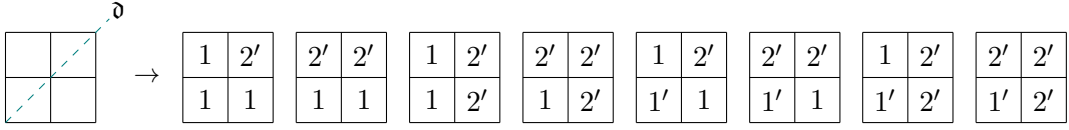

\begin{defn}We define the folded P- and Q-functions as follows.
$$Q^f_{(\gamma,\mathfrak{d})}(X)=\sum_{S\in \emph{SSFT}(\gamma,\mathfrak{d})}X^{|S|}$$
$$P^f_{(\gamma,\mathfrak{d})}(X)=2^{-size(\mathfrak{d})}Q^f_{(\gamma,\mathfrak{d})}(X)$$
where $SSFT(\gamma,\mathfrak{d})$ denotes the set of all semi-standard folded tableaux of shape $\Gamma=(\gamma,\mathfrak{d})$.
\end{defn}

The folded shape $((2,2),\mathfrak{d})$ illusturated in Figures \ref{fig:foldst} and \ref{fig:foldss} has the following folded P- and Q-functions:

$P^f_{(22,\mathfrak{d})}(X)=m_{31}(X)+2m_{22}(X)+4m_{211}(X)+8m_{1111}(X)= s_{31}(X)+s_{22}(X)+s_{2111}(X)=P_{(31)}(X)$

$Q^f_{(22,\mathfrak{d})}(X)=2^{ht(\mathfrak{d})}P^f_{(22,\mathfrak{d})}(X)=4(s_{31}(X)+s_{22}(X)+s_{2111}(X))=Q_{(31)}(X)$

\begin{thm}The function $Q^f_{(\gamma,\mathfrak{d})}(X)$ is independent of the choice of the main diagonal $\mathfrak{d}$.
\end{thm}
\begin{proof} We will prove this by giving a bijection weight preserving between SSFT($\gamma,\mathfrak{d}$) with semi-standard fillings of the skew-shifted shape $\gamma$. 

Let $S\in \emph{SSFT}(\gamma,\mathfrak{d})$. We start with inverting the markings of all cells weakly above the main diagonal. Each connected piece weakly above the main diagonal containing only $i$ or $i'$s forms a ribbon, with at most one $i'$ on each row (on the rightmost cell), and at most one $i$ on each column (on the lowest cell). The next step is to go from bottom to the top in the ribbon, flipping $i'$ with the leftmost $i$ of each row and flipping the $i$ with the highest $i'$ on each column.  As the algorithm corrects the ordering when the process is repeated for all $i$, we get with a semi-standard skew-shifted filling of $\gamma$.
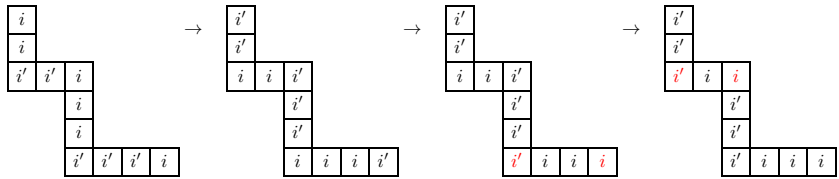
\begin{figure}[H]
\scalebox{0.7}{
\centering
$\begin{ytableau}
i\\
i\\
i'&i'&i\\
\none&\none&i\\
\none&\none&i\\
\none&\none&i'&i'&i'&i\\
\end{ytableau}\rightarrow\quad\begin{ytableau}
i'\\
i'\\
i&i&i'\\
\none&\none&i'\\
\none&\none&i'\\
\none&\none&i&i&i&i'\\
\end{ytableau}\rightarrow\quad\begin{ytableau}
i'\\
i'\\
i&i&i'\\
\none&\none&i'\\
\none&\none&i'\\
\none&\none&\color{red} i'&i&i&\color{red} i\\
\end{ytableau}\rightarrow\quad\begin{ytableau}
i'\\
i'\\
\color{red}i'&i&\color{red}i\\
\none&\none&i'\\
\none&\none&i'\\
\none&\none& i'&i&i& i\\
\end{ytableau}$}
\label{fig:iribbon}
\end{figure}

The process can be inverted by inverting the markings weakly above the main diagonal again, and this time working our way from top to bottom on each ribbon, flipping any $i'$ with the lowest $i$ on columns and flipping any $i$ with the rightmost $i'$ in rows. 
\end{proof}

\begin{cor}The folded Q-functions are Schur Q-positive. In fact 
$$Q^f_{(\gamma,\mathfrak{d})}(X)=Q_{\lambda/\delta_n}(X)=\sum_{\epsilon}f^{\lambda}_{\epsilon,\delta_n}Q_{\epsilon}(X) $$
where $\lambda$ is a shifted shape with $\lambda/\delta_n=\gamma$ and $f^{\lambda}_{\epsilon,\delta_n}$ are the non-negative integers defined by:
$$P_{\epsilon}P_{\delta_n}=\sum_{\lambda}f^{\lambda}_{\epsilon,\delta_n}P_{\lambda}$$
\end{cor}

As the folded $P$ function depends on the size of the main diagonal, it is not independent of $\mathfrak{d}$. Instead, it tells us that  $Q^f_{\gamma}(X)$ is divisible by $2^d$, where $d$ is the size of the longest diagonal on $\gamma$.

\begin{cor}An unshifted shape and its conjugate have the same folded Q-function. In particular, for two shifted shapes $\alpha$ and $\beta$, $Q^f_{\alpha \diamond \beta}(X)=Q^f_{\beta \diamond \alpha}(X)$
\end{cor}
\begin{proof}
For an unshifted shape $\gamma$, the conjugation operation gives a bijection of folded tableaux $(\gamma,\mathfrak{d})$ with $\mathfrak{d}$ above the shape and folded tableaux $(\gamma^T,\mathfrak{d}')$ with $\mathfrak{d}'$ below the shape. As the folded Q-function is independent of the placement of the specialized diagonal, we have: $Q^f_{\gamma}(X)=Q^f_{\gamma^T}(X)$
\end{proof}

\section{Quotients of Ribbon Tableaux}
\label{sec:quotients}

On this section, we will introduce the $k$-quotient for a shifted diagram, and we will give a bijection between the $k$-ribbon tableaux and the fillings of its $k$-quotient. Our  definition of the $k$-quotient extends the one given by Morris and Yaseen in\cite{MR809494} by specialized diagonals which we will use for a direct bijection between semi-standard $k$-ribbon tableaux and semi-standard fillings of its $k$-quotient.

\begin{defn} The $k$-quotient of a shifted shape $\lambda$ with $k$-abacus representation $(\alpha^{(1)},\alpha^{(2)},\ldots,\alpha^{(k)})$ will consist of a $\lfloor k/2\rfloor$ folded shapes and one shifted shape, defined as follows: 

$\overline{\Phi^k}(\lambda)=
\begin{cases}(\alpha^{(1)} \diamond \alpha^{(k-1)},\alpha^{(2)} \diamond \alpha^{(k-2)},\ldots,\alpha^{((k-1)/2)} \diamond \alpha^{((k+1)/2)},\alpha^{(k)})& \text{k odd}\\

(\alpha^{(1)} \diamond \alpha^{(k-1)},\alpha^{(2)} \diamond \alpha^{(k-2)},\ldots,\alpha^{(k/2-1)} \diamond \alpha^{(k/2+1)},\alpha^{(k/2)} \diamond \varnothing ,\alpha^{(k)})& \text{k even}
\end{cases}$
\end{defn}

The strict partition $\lambda=\{16, 11, 10, 9, 8, 7, 4, 3, 1\}$ has the $5$-quotient:
$$\overline{\Phi^5}(\lambda)=\{(4,3,1)\diamond(2,1),(2)\diamond(2,1),(2)\}=\{((3,3,2),\mathfrak{d}_1),((3),\mathfrak{d}_2) ,(2)\}$$
where $\mathfrak{d}_1$ and $\mathfrak{d}_2$ are the specialized diagonals given by the combination operation.

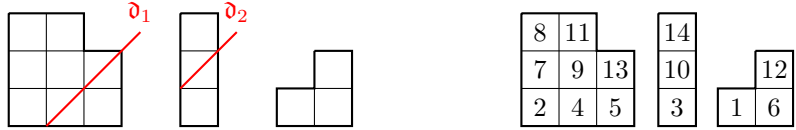
\begin{figure}[ht]
\centering
\raisebox{-.6cm}{
\begin{tikzpicture}[scale=1/2]
\draw[thick] (0,0)--(2,0)--(2,-1)--(3,-1)--(3,-3)--(0,-3)--(0,0);
\draw (0,-1)--(2,-1) (0,-2)--(3,-2);
\draw (1,0)--(1,-3) (2,-1)--(2,-3);
\draw[red, thick] (1,-3)--(3.5,-0.5) node[above, red] {$\mathfrak{d}_1$};
\end{tikzpicture}}
\raisebox{-0.6cm}{
\begin{tikzpicture}[scale=1/2]
\draw[thick] (0,0)--(1,0)--(1,-3)--(0,-3)--(0,0);
\draw (0,-1)--(1,-1) (0,-2)--(1,-2);
\draw[red,thick] (0,-2)--(1.5,-0.5) node[above, red] {$\mathfrak{d}_2$};
\end{tikzpicture}}
\raisebox{-0.6cm}{
\begin{tikzpicture}[scale=1/2]
\draw[thick] (0,0)--(1,0)--(1,-2)--(-1,-2)--(-1,-1)--(0,-1)--(0,0);
\draw (0,-2)--(0,-1)--(1,-1);
\end{tikzpicture}}\qquad \qquad \qquad
\raisebox{-0.6cm}{
\begin{tikzpicture}[scale=1/2]
\draw[thick] (0,0)--(2,0)--(2,-1)--(3,-1)--(3,-3)--(0,-3)--(0,0);
\draw (0,-1)--(2,-1) (0,-2)--(3,-2);
\draw (1,0)--(1,-3) (2,-1)--(2,-3);

\node  at (0.5,-0.5) {8};
\node  at (1.5,-0.5) {11};
\node  at (0.5,-1.5) {7};
\node  at (1.5,-1.5) {9};
\node  at (2.5,-1.5) {13};
\node  at (0.5,-2.5) {2};
\node  at (1.5,-2.5) {4};
\node  at (2.5,-2.5) {5};
\end{tikzpicture}}
\raisebox{-0.6cm}{
\begin{tikzpicture}[scale=1/2]
\draw[thick] (0,0)--(1,0)--(1,-3)--(0,-3)--(0,0);
\draw (0,-1)--(1,-1) (0,-2)--(1,-2);
\node at (0.5,-0.5) {14};
\node at (0.5,-1.5) {10};
\node at (0.5,-2.5) {3};
\end{tikzpicture}}
\raisebox{-0.6cm}{
\begin{tikzpicture}[scale=1/2]
\draw[thick] (0,0)--(1,0)--(1,-2)--(-1,-2)--(-1,-1)--(0,-1)--(0,0);
\draw (0,-2)--(0,-1)--(1,-1);
\node at (0.5,-0.5) {12};
\node at (0.5,-1.5) {6};
\node at (-0.5,-1.5) {1};
\end{tikzpicture}}
\label{fig:bigq}
\caption{The $5$-quotient of $\lambda=(16, 11, 10, 9, 8, 7, 4, 3, 1)$ and a standard filling.}
\end{figure}

\ytableausetup{boxsize=4mm}

\begin{figure}[p]
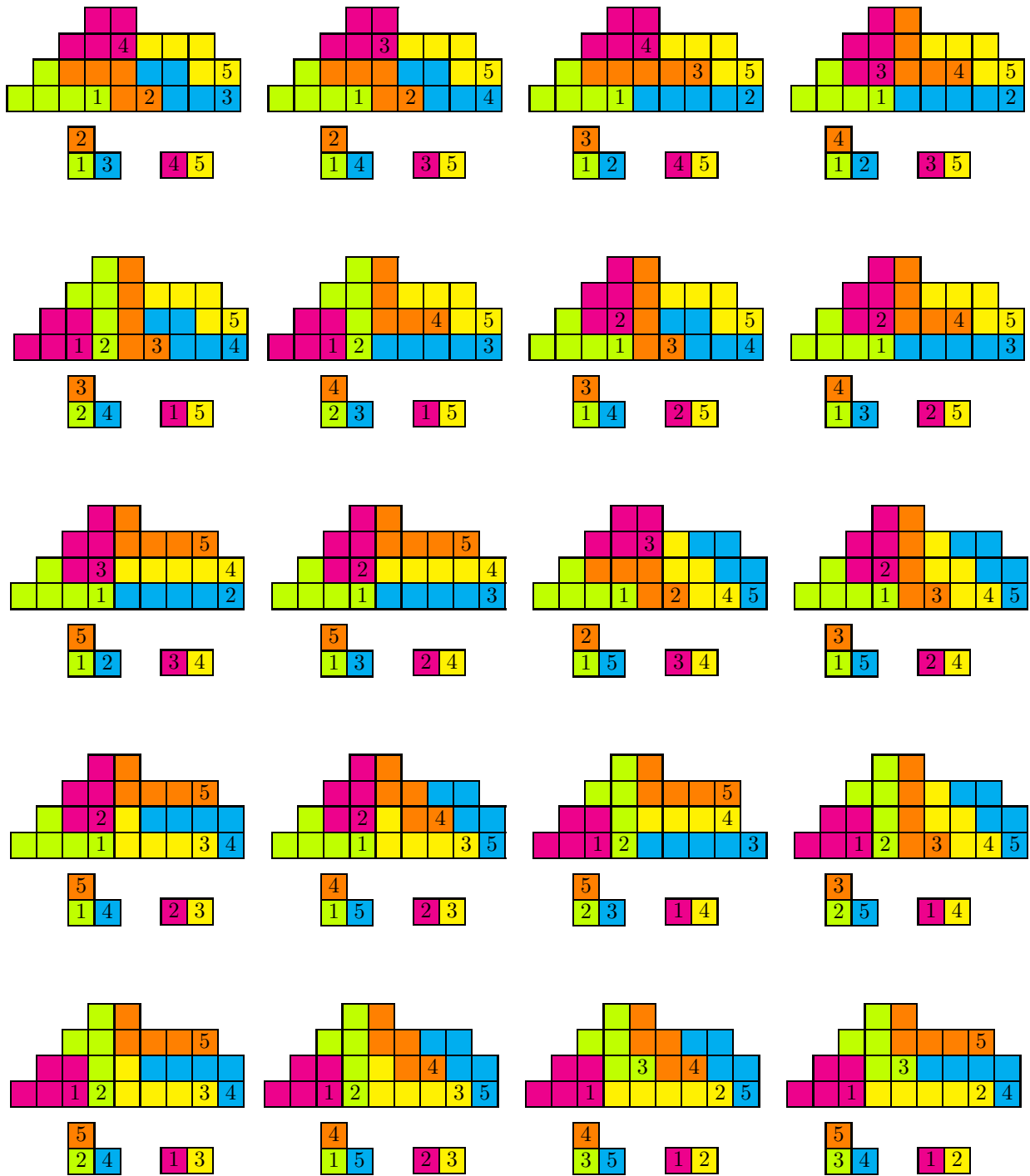
\label{fig:fullpage}
 \begin{ytableau}
   \none & \none & \none & *(magenta) & *(magenta)  \\
   \none &\none & *(magenta) & *(magenta)  & *(magenta)4 &*(yellow)&*(yellow) &*(yellow)\\
   \none &*(lime) &*(orange)&*(orange)&*(orange)&*(cyan)&*(cyan) &*(yellow)&*(yellow)5\\
   *(lime) &*(lime)&*(lime)&*(lime)1&*(orange)&*(orange)2&*(cyan)&*(cyan)&*(cyan)3\\
  \end{ytableau} \hspace{0.2cm} \begin{ytableau}
      \none & \none & \none & *(magenta) & *(magenta)  \\
   \none &\none & *(magenta) & *(magenta)  & *(magenta)3 &*(yellow)&*(yellow) &*(yellow)\\
   \none &*(lime) &*(orange)&*(orange)&*(orange)&*(cyan)&*(cyan) &*(yellow)&*(yellow)5\\
   *(lime) &*(lime)&*(lime)&*(lime)1&*(orange)&*(orange)2&*(cyan)&*(cyan)&*(cyan)4\\
  \end{ytableau} \hspace{0.2cm} \begin{ytableau}
     \none & \none & \none & *(magenta) & *(magenta)  \\
   \none &\none & *(magenta) & *(magenta)  & *(magenta)4 &*(yellow)&*(yellow) &*(yellow)\\
   \none &*(lime) &*(orange)&*(orange)&*(orange)&*(orange)&*(orange)3&*(yellow)&*(yellow)5\\
   *(lime) &*(lime)&*(lime)&*(lime)1&*(cyan)&*(cyan)&*(cyan)&*(cyan)&*(cyan)2\\
  \end{ytableau}  \hspace{0.2cm} \begin{ytableau}
   \none & \none & \none & *(magenta) & *(orange)  \\
   \none &\none & *(magenta) & *(magenta)& *(orange)&*(yellow)&*(yellow)&*(yellow)\\
      \none &*(lime) & *(magenta) & *(magenta)3 & *(orange)& *(orange)& *(orange)4&*(yellow)&*(yellow)5\\
 *(lime) &*(lime)&*(lime)&*(lime)1&*(cyan)&*(cyan)&*(cyan)&*(cyan)&*(cyan)2\\
  \end{ytableau}  
  \vspace{0.2cm}
  
  \begin{ytableau}
  *(orange)2\\
  *(lime)1&*(cyan)3\\
  \end{ytableau} \hspace{4mm} \raisebox{-0.4cm}{\begin{ytableau}
  *(magenta)4&*(yellow)5\\
  \end{ytableau}}
  \hspace{14.8mm} \begin{ytableau}
  *(orange)2\\
  *(lime)1&*(cyan)4\\
  \end{ytableau} \hspace{4mm} \raisebox{-0.4cm}{\begin{ytableau}
  *(magenta)3&*(yellow)5\\
  \end{ytableau}}
    \hspace{14.8mm} \begin{ytableau}
  *(orange)3\\
  *(lime)1&*(cyan)2\\
  \end{ytableau} \hspace{4mm} \raisebox{-0.4cm}{\begin{ytableau}
  *(magenta)4&*(yellow)5\\
  \end{ytableau}}  \hspace{14.8mm} \begin{ytableau}
  *(orange)4\\
  *(lime)1&*(cyan)2\\
  \end{ytableau} \hspace{4mm} \raisebox{-0.4cm}{\begin{ytableau}
  *(magenta)3&*(yellow)5\\
  \end{ytableau}}
  
  \vspace{1.2cm}

  \begin{ytableau}
   \none & \none & \none &*(lime)&*(orange) \\
   \none &\none &*(lime)&*(lime)&*(orange) &*(yellow)&*(yellow)&*(yellow)\\
   \none &*(magenta)&*(magenta)&*(lime)&*(orange) &*(cyan)&*(cyan)&*(yellow)&*(yellow)5\\
  *(magenta)&*(magenta)&*(magenta)1&*(lime)2&*(orange) &*(orange)3 &*(cyan)&*(cyan)&*(cyan)4\\
  \end{ytableau} \hspace{0.2cm}\begin{ytableau}
\none & \none & \none &*(lime)&*(orange) \\
   \none &\none &*(lime)&*(lime)&*(orange)&*(yellow)&*(yellow)&*(yellow) \\
   \none &*(magenta)&*(magenta)&*(lime)&*(orange) &*(orange) &*(orange) 4&*(yellow)&*(yellow)5\\
 *(magenta)&*(magenta)&*(magenta)1&*(lime)2&*(cyan)&*(cyan)&*(cyan)&*(cyan)&*(cyan)3\\
  \end{ytableau}  \hspace{0.2cm} \begin{ytableau}
     \none & \none & \none &*(magenta)&*(orange) \\
   \none &\none &*(magenta)&*(magenta)&*(orange) &*(yellow) &*(yellow) &*(yellow) \\
   \none &*(lime) &*(magenta)&*(magenta)2&*(orange)&*(cyan)&*(cyan) &*(yellow)&*(yellow)5\\
 *(lime) &*(lime) &*(lime) &*(lime) 1&*(orange) &*(orange) 3&*(cyan)&*(cyan)&*(cyan)4 \\
  \end{ytableau} \hspace{0.2cm} \begin{ytableau}
     \none & \none & \none &*(magenta)&*(orange) \\
   \none &\none &*(magenta)&*(magenta)&*(orange)  &*(yellow) &*(yellow) &*(yellow)\\
   \none &*(lime) &*(magenta)&*(magenta)2&*(orange) &*(orange) &*(orange) 4 &*(yellow) &*(yellow)5\\
  *(lime) &*(lime) &*(lime) &*(lime) 1&*(cyan)&*(cyan)&*(cyan)&*(cyan)&*(cyan)3\\
  \end{ytableau}

    \vspace{0.2cm}
  
  \begin{ytableau}
  *(orange)3\\
  *(lime)2&*(cyan)4\\
  \end{ytableau} \hspace{4mm} \raisebox{-0.4cm}{\begin{ytableau}
  *(magenta)1&*(yellow)5\\
  \end{ytableau}}
  \hspace{14.8mm} \begin{ytableau}
  *(orange)4\\
  *(lime)2&*(cyan)3\\
  \end{ytableau} \hspace{4mm} \raisebox{-0.4cm}{\begin{ytableau}
  *(magenta)1&*(yellow)5\\
  \end{ytableau}}
    \hspace{14.8mm} \begin{ytableau}
  *(orange)3\\
  *(lime)1&*(cyan)4\\
  \end{ytableau} \hspace{4mm} \raisebox{-0.4cm}{\begin{ytableau}
  *(magenta)2&*(yellow)5\\
  \end{ytableau}}  \hspace{14.8mm} \begin{ytableau}
  *(orange)4\\
  *(lime)1&*(cyan)3\\
  \end{ytableau} \hspace{4mm} \raisebox{-0.4cm}{\begin{ytableau}
  *(magenta)2&*(yellow)5\\
  \end{ytableau}}
  
  \vspace{1.2cm}

  \begin{ytableau}
    \none & \none & \none &*(magenta)&*(orange)\\
   \none &\none &*(magenta)&*(magenta)&*(orange)&*(orange)&*(orange)&*(orange)5\\
   \none  &*(lime) &*(magenta)&*(magenta)3&*(yellow)&*(yellow)&*(yellow)&*(yellow)&*(yellow)4\\
  *(lime) &*(lime) &*(lime) &*(lime)1&*(cyan)&*(cyan)&*(cyan)&*(cyan)&*(cyan)2 \\
  \end{ytableau} \hspace{0.2cm} \begin{ytableau}
      \none & \none & \none &*(magenta)&*(orange)\\
   \none &\none &*(magenta)&*(magenta)&*(orange)&*(orange)&*(orange)&*(orange)5\\
    \none  &*(lime)&*(magenta)&*(magenta)2 &*(yellow)&*(yellow)&*(yellow)&*(yellow)&*(yellow)4\\
  *(lime) &*(lime) &*(lime) &*(lime)1 &*(cyan)&*(cyan)&*(cyan)&*(cyan)&*(cyan)3\\
  \end{ytableau} \hspace{0.2cm} \begin{ytableau}
  \none & \none & \none &*(magenta)&*(magenta)\\
  \none &\none &*(magenta)&*(magenta)&*(magenta)3&*(yellow)&*(cyan)&*(cyan)\\
    \none  &*(lime) &*(orange)&*(orange)&*(orange)&*(yellow)&*(yellow)&*(cyan)&*(cyan)\\
  *(lime) &*(lime) &*(lime) &*(lime)1 &*(orange)&*(orange)2&*(yellow)&*(yellow)4&*(cyan)5\\
  \end{ytableau} \hspace{0.2cm} \begin{ytableau}
   \none & \none & \none &*(magenta)&*(orange)\\
   \none &\none &*(magenta)&*(magenta)&*(orange)&*(yellow)&*(cyan)&*(cyan)\\
    \none  &*(lime) &*(magenta)&*(magenta)2&*(orange)&*(yellow)&*(yellow)&*(cyan)&*(cyan)\\
  *(lime) &*(lime) &*(lime) &*(lime)1&*(orange)&*(orange)3 &*(yellow)&*(yellow)4&*(cyan)5\\
  \end{ytableau} 
  
   \vspace{0.2cm}
  
  \begin{ytableau}
  *(orange)5\\
  *(lime)1&*(cyan)2\\
  \end{ytableau} \hspace{4mm} \raisebox{-0.4cm}{\begin{ytableau}
  *(magenta)3&*(yellow)4\\
  \end{ytableau}}
  \hspace{14.8mm} \begin{ytableau}
  *(orange)5\\
  *(lime)1&*(cyan)3\\
  \end{ytableau} \hspace{4mm} \raisebox{-0.4cm}{\begin{ytableau}
  *(magenta)2&*(yellow)4\\
  \end{ytableau}}
    \hspace{14.8mm} \begin{ytableau}
  *(orange)2\\
  *(lime)1&*(cyan)5\\
  \end{ytableau} \hspace{4mm} \raisebox{-0.4cm}{\begin{ytableau}
  *(magenta)3&*(yellow)4\\
  \end{ytableau}}  \hspace{14.8mm} \begin{ytableau}
  *(orange)3\\
  *(lime)1&*(cyan)5\\
  \end{ytableau} \hspace{4mm} \raisebox{-0.4cm}{\begin{ytableau}
  *(magenta)2&*(yellow)4\\
  \end{ytableau}}
  
  \vspace{1.2cm}
   
   
    \begin{ytableau}
      \none & \none & \none&*(magenta)  &*(orange)\\
   \none &\none&*(magenta)&*(magenta) &*(orange) &*(orange) &*(orange) &*(orange)5\\
   \none  &*(lime)&*(magenta)&*(magenta)2  &*(yellow)&*(cyan)&*(cyan)&*(cyan)&*(cyan)\\
  *(lime) &*(lime) &*(lime) &*(lime)1  &*(yellow) &*(yellow) &*(yellow) &*(yellow)3&*(cyan)4\\
  \end{ytableau} \hspace{0.2cm} \begin{ytableau}
   \none & \none & \none&*(magenta)  &*(orange)\\
   \none &\none&*(magenta)&*(magenta) &*(orange) &*(orange) &*(cyan)&*(cyan)\\
   \none  &*(lime) &*(magenta)&*(magenta)2 &*(yellow) &*(orange) &*(orange)4&*(cyan)&*(cyan)\\
  *(lime) &*(lime) &*(lime) &*(lime)1  &*(yellow) &*(yellow) &*(yellow) &*(yellow)3&*(cyan)5\\
  \end{ytableau} \hspace{0.2cm} \begin{ytableau}
     \none & \none & \none &*(lime) &*(orange)\\
   \none &\none &*(lime)&*(lime) &*(orange) &*(orange) &*(orange) &*(orange)5\\
   \none&*(magenta)&*(magenta)&*(lime)  &*(yellow) &*(yellow) &*(yellow) &*(yellow)4\\
 *(magenta)&*(magenta)&*(magenta)1 &*(lime)2&*(cyan)&*(cyan)&*(cyan)&*(cyan)&*(cyan)3\\
  \end{ytableau}  \hspace{0.2cm} \begin{ytableau}
   \none & \none & \none&*(lime)  &*(orange)\\
   \none &\none &*(lime)&*(lime) &*(orange)&*(yellow)&*(cyan)&*(cyan)\\
   \none&*(magenta)&*(magenta) &*(lime) &*(orange)&*(yellow)&*(yellow)&*(cyan)&*(cyan)\\
 *(magenta)&*(magenta)&*(magenta)1 &*(lime)2 &*(orange) &*(orange)3&*(yellow)&*(yellow)4&*(cyan)5\\
  \end{ytableau}  

  \vspace{0.2cm}
  \begin{ytableau}
  *(orange)5\\
  *(lime)1&*(cyan)4\\
  \end{ytableau} \hspace{4mm} \raisebox{-0.4cm}{\begin{ytableau}
  *(magenta)2&*(yellow)3\\
  \end{ytableau}}
  \hspace{14.8mm} \begin{ytableau}
  *(orange)4\\
  *(lime)1&*(cyan)5\\
  \end{ytableau} \hspace{4mm} \raisebox{-0.4cm}{\begin{ytableau}
  *(magenta)2&*(yellow)3\\
  \end{ytableau}}
    \hspace{14.8mm} \begin{ytableau}
  *(orange)5\\
  *(lime)2&*(cyan)3\\
  \end{ytableau} \hspace{4mm} \raisebox{-0.4cm}{\begin{ytableau}
  *(magenta)1&*(yellow)4\\
  \end{ytableau}}  \hspace{14.8mm} \begin{ytableau}
  *(orange)3\\
  *(lime)2&*(cyan)5\\
  \end{ytableau} \hspace{4mm} \raisebox{-0.4cm}{\begin{ytableau}
  *(magenta)1&*(yellow)4\\
  \end{ytableau}}
  
  \vspace{1.2cm}

    \begin{ytableau}
    \none & \none & \none &*(lime) &*(orange)\\
   \none &\none&*(lime)&*(lime)  &*(orange) &*(orange) &*(orange) &*(orange)5\\
   \none&*(magenta)&*(magenta) &*(lime) &*(yellow)&*(cyan)&*(cyan)&*(cyan)&*(cyan)\\
 *(magenta)&*(magenta)&*(magenta)1 &*(lime)2 &*(yellow) &*(yellow) &*(yellow) &*(yellow)3&*(cyan)4\\
  \end{ytableau}\hspace{0.2cm}  \begin{ytableau}
   \none & \none & \none &*(lime) &*(orange)\\
   \none &\none &*(lime)&*(lime) &*(orange) &*(orange)&*(cyan)&*(cyan)\\
   \none&*(magenta)&*(magenta)&*(lime)  &*(yellow) &*(orange) &*(orange)4&*(cyan)&*(cyan)\\
 *(magenta)&*(magenta)&*(magenta)1&*(lime)2  &*(yellow) &*(yellow) &*(yellow) &*(yellow)3&*(cyan)5\\
  \end{ytableau} \hspace{0.2cm} \begin{ytableau}
   \none & \none & \none &*(lime) &*(orange)\\
   \none &\none &*(lime)&*(lime) &*(orange) &*(orange)&*(cyan)&*(cyan)\\
   \none&*(magenta)&*(magenta) &*(lime)&*(lime)3 &*(orange) &*(orange)4&*(cyan)&*(cyan)\\
 *(magenta)&*(magenta)&*(magenta)1 &*(yellow) &*(yellow) &*(yellow) &*(yellow) &*(yellow)2 &*(cyan)5\\
  \end{ytableau}  \hspace{0.2cm} \begin{ytableau}
     \none & \none & \none &*(lime) &*(orange)\\
   \none &\none &*(lime)&*(lime) &*(orange) &*(orange) &*(orange) &*(orange)5\\
   \none&*(magenta)&*(magenta)&*(lime)&*(lime)3 &*(cyan)&*(cyan)&*(cyan)&*(cyan)\\
 *(magenta)&*(magenta)&*(magenta)1  &*(yellow) &*(yellow) &*(yellow) &*(yellow) &*(yellow)2&*(cyan)4\\
  \end{ytableau}  
  \vspace{0.2cm}

 \begin{ytableau}
  *(orange)5\\
  *(lime)2&*(cyan)4\\
  \end{ytableau} \hspace{4mm} \raisebox{-0.4cm}{\begin{ytableau}
  *(magenta)1&*(yellow)3\\
  \end{ytableau}}
  \hspace{14.8mm} \begin{ytableau}
  *(orange)4\\
  *(lime)1&*(cyan)5\\
  \end{ytableau} \hspace{4mm} \raisebox{-0.4cm}{\begin{ytableau}
  *(magenta)2&*(yellow)3\\
  \end{ytableau}}
    \hspace{14.8mm} \begin{ytableau}
  *(orange)4\\
  *(lime)3&*(cyan)5\\
  \end{ytableau} \hspace{4mm} \raisebox{-0.4cm}{\begin{ytableau}
  *(magenta)1&*(yellow)2\\
  \end{ytableau}}  \hspace{14.8mm} \begin{ytableau}
  *(orange)5\\
  *(lime)3&*(cyan)4\\
  \end{ytableau} \hspace{4mm} \raisebox{-0.4cm}{\begin{ytableau}
  *(magenta)1&*(yellow)2\\
  \end{ytableau}}
  
  \vspace{1.2cm}
 \caption{Correspondence between the standard $5$-ribbon tableaux of shape $(9,8,6,2)$ and the standard fillings of its quotient$\{(2,1),(2),\varnothing\}$}
 \end{figure}
 \ytableausetup{nosmalltableaux}

We call  a simultaneous semi-standard filling of the $\lfloor k/2\rfloor$ folded shapes and the shifted shape a \emph{semi-standard filling} of the $k$-quotient. If this filling uses each number from $1$ to $n$ exactly once, unmarked, it will be called a \emph{standard filling}.

\begin{thm} There is a bijection $\Phi^k_{\lambda}$ between standard $k$-ribbon tableaux of shape $\lambda$ and standard fillings of its $k$-quotient preserving diagonal values (two ribbons that have the same diagonal value will be mapped to the same shape and diagonal in the quotient). 
\end{thm}

\begin{figure}[ht]
\centering

\begin{tikzpicture}[scale=7/15]
\draw[thick] (0,0)--(2,0)--(2,-1)--(3,-1)--(3,-3)--(0,-3)--(0,0);
\draw (0,-1)--(2,-1) (0,-2)--(3,-2);
\draw (1,0)--(1,-3) (2,-1)--(2,-3);

\node at (0.5,-0.5) {8};
\node at (1.5,-0.5) {11};
\node at (0.5,-1.5) {7};
\node at (1.5,-1.5) {9};
\node at (2.5,-1.5) {13};
\node at (0.5,-2.5) {2};
\node at (1.5,-2.5) {4};
\node at (2.5,-2.5) {5};
\end{tikzpicture}\quad
\begin{tikzpicture}[scale=4/7]
\draw[thick] (0,0)--(1,0)--(1,-3)--(0,-3)--(0,0);
\draw (0,-1)--(1,-1) (0,-2)--(1,-2);

\node at (0.5,-0.5) {14};
\node at (0.5,-1.5) {10};
\node at (0.5,-2.5) {3};

\end{tikzpicture}\quad
\begin{tikzpicture}[scale=4/7]
\draw[thick] (0,0)--(1,0)--(1,-2)--(-1,-2)--(-1,-1)--(0,-1)--(0,0);
\draw (0,-2)--(0,-1)--(1,-1);

\node at (0.5,-0.5) {12};
\node at (0.5,-1.5) {6};
\node at (-0.5,-1.5) {1};
\end{tikzpicture}\qquad
\begin{tikzpicture}[scale=3/7]
\draw[thick] (0,0)--(0,-1)--(-1,-1)--(-1,-2)--(-2,-2)--(-2,-3)--(-3,-3)--(-3,-4)--(-4,-4)--(-4,-5)--(-5,-5)--(-5,-6)--(-6,-6)--(-6,-7)--(-7,-7)--(-7,-8)--(-8,-8)--(-8,-9)--(-9,-9)--(-9,-10)--(7,-10)--(7,-9)--(3,-9)--(3,-4)--(2,-4)--(2,-1)--(1,-1)--(1,0)--(0,0);
\draw (1,-1)--(1,-6)--(2,-6)--(2,-9)--(3,-9);
\draw (3,-5)--(-1,-5)--(-1,-3)--(-2,-3) (2,-10)--(2,-9) (-1,-2)--(0,-2)--(0,-4)--(1,-4);
\draw (-1,-5)--(-4,-5)--(-4,-6)--(1,-6) (0,-6)--(0,-7)--(1,-7)--(1,-10);
\draw (-1,-6)--(-1,-8)--(0,-8)--(0,-10) (-2,-6)--(-2,-9)--(-1,-9)--(-1,-10) (-6,-7)--(-2,-7);
\draw (-4,-7)--(-4,-8)--(-3,-8)--(-3,-10) (-6,-7)--(-6,-8)--(-5,-8)--(-5,-9)--(-4,-9)--(-4,-10);
\draw (-7,-8)--(-7,-9)--(-6,-9)--(-6,-10);
\node at (2.5,-4.5) {14};
\node at (0.5,-3.5) {13};
\node   at (0.5,-4.5) {12};
\node  at (2.5,-8.5) {11};
\node  at (-1.5,-4.5) {10};
\node at (0.5,-5.5) {9};
\node   at (6.5,-9.5) {8};
\node  at (1.5,-9.5) {7};
\node  at (0.5,-9.5) {6};
\node   at (-0.5,-9.5) {5};
\node  at (-2.5,-6.5) {4};
\node  at (-1.5,-9.5) {3};
\node   at (-3.5,-9.5) {2};
\node  at (-4.5,-9.5) {1};
\end{tikzpicture}
\caption{A $5$-ribbon tableau of shape $\lambda=(16,11,10,9,8,7,4,3,1)$ with the corresponding filling of the $5$-quotient}
\end{figure}
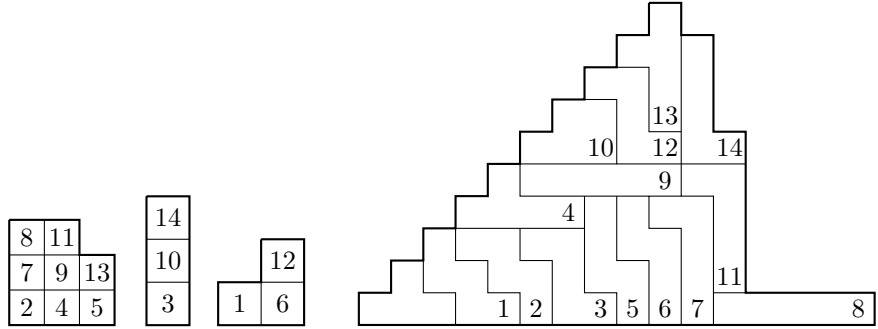

\begin{proof} Consider a $k$-ribbon tableau $T$ of shape $\lambda$ with abacus representation $(\alpha^{(1)},\alpha^{(2)},\ldots,\alpha^{(k)})$. As each ribbon corresponds to a move in the abacus representation of $\lambda$, $T$ uniquely corresponds to a sequence of moves from $(\alpha^{(1)},\alpha^{(2)}..\alpha^{(k)})$ to the $k$-core of $\lambda$. As we can move independently on each runner pair $(a_i,a_{k-i})$ and on $a_k$, it will suffice to match the moves on $a_k$ moves to shifted tableaux of shape $\alpha^{(k)}$, and the moves on runner pairs $(a_i,a_{k-i})$ to moves on $\alpha^{(i)}\diamond \alpha^{(k-i)}$for each $i$. 
\begin{claimm} There is a bijection between sequences of moves from $a_k$ to the empty runner and standard shifted tableaux of shape $\alpha^{(k)}$, where a move from row $d$ to $d-1$ on the abacus corresponds to a box on diagonal $d$.
\end{claimm}

\begin{figure}[ht]
    \centering
\begin{tabular}{ l }
  5\\
  \hline
 $\bullet$\\
$\bullet$\\
 $\circ$\\
\end{tabular}$\qquad \rightarrow 3 \qquad $
\begin{tabular}{ l }
  5\\
  \hline
 $\circ$\\
$\bullet$\\
 $\circ$\\
\end{tabular}$\qquad  \rightarrow 2 \qquad $
\begin{tabular}{ l }
  5\\
  \hline
 $\bullet$\\
$\circ$\\
 $\circ$\\
\end{tabular}$\qquad \rightarrow 1 \qquad $
\begin{tabular}{ l }
  5\\
  \hline
 $\circ$\\
$\circ$\\
 $\circ$\\
\end{tabular}$\qquad \qquad \qquad $
\raisebox{-0.6cm}
{\begin{tikzpicture}[scale=1/2]
\draw (0,0) -- (1,0);
\draw (-1,-1) -- (1,-1);
\draw (-1,-2) -- (1,-2);
\draw (0,0) -- (0,-2)   (1,0) -- (1,-2) ;
\draw (-1,-1) -- (-1,-2)  ;
\draw[dashed] (1,0) -- (1.5,0.5) node {$\quad$1};
\draw[dashed] (1,-1) -- (1.5,-0.5) node {$\quad$2};
\node at (0.5,-0.5) {3};
\node at (0.5,-1.5) {2};
\node at (-0.5,-1.5) {1};
\end{tikzpicture}}
\caption{Moves on runner $a_5$ can be matched to a standard filling of the shifted diagram $\alpha^{(5)}=(2,1)$}
\label{fig:runib}
\end{figure}
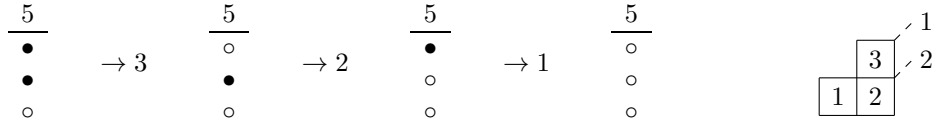

\begin{claimmproof} A bead on the $i$th row of runner $a_k$ will make a total of $i$ moves, $i-1$ to one row higher, and one last move to be removed. We map these moves to a row of $i$ boxes so that the move from position $j$ to $j-1$ will correspond to a box on diagonal $j$, and the removal move will correspond a box on the main diagonal. This means if $a_k$ has beads on positions $i_1> i_2>\cdots >i_t$ we will map the moves to the shifted diagram $\alpha^{(k)}=(i_1,i_2,\ldots,i_t)$. Let us number the moves in decreasing order with numbers from $1$ to $|\alpha^{(k)}|$. This will give us a filling of $\alpha^{(k)}$, with the conditions that beads can only move to unoccupied positions (meaning rows need to increase left to right), and a bead can only move higher (meaning columns increase bottom to top). Note that these conditions exactly correspond to the tableaux conditions.
\end{claimmproof}

Now, we can turn our attention to runner pairs $a_i$, $a_{k-i}$.
\begin{claimm} There is a bijection between sequences of moves from the pair of runners $(a_i,a_{k-i})$ to the abacus core and standard unshifted tableau of shape $\alpha^{(i)}\diamond \alpha^{(k-i)}$, where moves removing beads are mapped to the specialized diagonal of $\alpha^{(i)}\diamond \alpha^{(k-i)}$, and moves on runner $a_i$ (resp. $a_{k-i}$) from row $r$ to $r-1$ are mapped to the diagonal $d$ units to the left (resp. right).
\end{claimm}
\begin{figure}[ht]
\begin{center}
\begin{tabular}{ l l }
  $a_1$ & $a_4$  \\
  \hline
$\bullet$ & $\bullet$\\
$\circ$ & $\bullet$\\
$\bullet$ & $\circ$ \\
$\bullet$ & $\circ$ \\
\end{tabular}$\quad \rightarrow 8 \quad $ 
\begin{tabular}{ l l }
$a_1$ & $a_4$ \\
  \hline
$\bullet$ & $\bullet$\\
$\bullet$ & $\bullet$\\
$\circ$ & $\circ$ \\
$\bullet$ & $\circ$ \\
\end{tabular}$\quad  \rightarrow 7 \quad $
\begin{tabular}{ l l }
$a_1$ & $a_4$\\
  \hline
$\circ$ & $\circ$\\
$\bullet$ & $\bullet$\\
$\circ$ & $\circ$ \\
$\bullet$ & $\circ$ \\
\end{tabular}$\quad \rightarrow 6 \quad $
\begin{tabular}{ l l }
  $a_1$ & $a_4$  \\
  \hline
$\bullet$ & $\circ$\\
$\circ$ & $\bullet$\\
$\circ$ & $\circ$ \\
$\bullet$ & $\circ$ \\
\end{tabular}$\quad \rightarrow 5 \quad $
$\quad \cdots  \quad $
\raisebox{-0.6cm}
{\begin{tikzpicture}[scale=1/2]
\draw (0,0) -- (2,0) (0,-2)--(3,-2) (0,-3)--(3,-3);
\draw (0,-1) -- (3,-1);
\draw (0,0) -- (0,-3) (1,0)--(1,-3) (2,0)--(2,-3) (3,-1)--(3,-3);
\draw[dashed] (3,-1) -- (3.5,-0.5) node[right, violet] {special diag.};
\draw[dashed, gray] (1,-3)--(3,-1);

\node at (1.5,-0.5) {8};
\node at (0.5,-0.5) {5};
\node at (1.5,-1.5) {6};
\node at (2.5,-1.5) {7};
\node at (2.5,-2.5) {4};
\node at (1.5,-2.5) {3};
\node at (0.5,-1.5) {2};
\node at (0.5,-2.5) {1};
\end{tikzpicture}}
\end{center}
\caption{Moves on runners $a_1$ and $a_4$ give a standard filling of the folded diagram $\alpha^{(1)} \diamond \alpha^{(1)}$}
\label{fig:runia}
\end{figure}
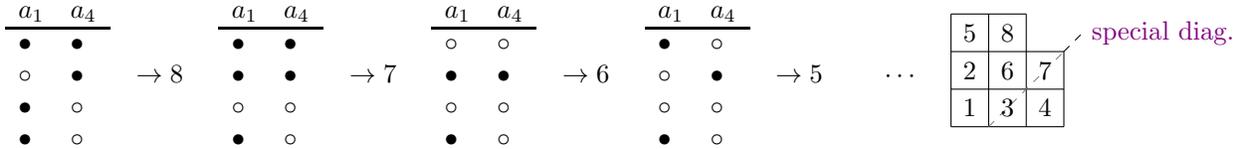
\begin{claimmproof} The move sequences on each runner can be matched to a shifted tableau of corresponding shape as in Claim 1, except now we have an additional constraint: To remove a bead from the first row one runner, we must simultaneously remove a bead from the first row of the second runner. This implies that the main diagonals of the two shapes must contain the same numbers, and they will be on the main diagonal. Furthermore, if one runner has $q$ more beads than the other one, these can not be removed, and the moves which depend on the removal of these beads can not be made, which means the the subdiagram of shape $\delta_q$ inside the larger shape will be left empty.
\end{claimmproof}
When $k$ is even, we can move ribbons up on runner $a_{k/2}$ but not remove them, as if it has a conjugate runner with no beads, so any remark we made on $\alpha^{(i)}\diamond \alpha^{(k-i)}$ above automatically applies to $\alpha^{(k/2)}\diamond \varnothing$.
\end{proof}

\begin{rmk}The correspondence of diagonals gives us a way of labeling the diagonals of the quotient to match the values of the original shape. This way, the shifted shape $\alpha^{(k)}$ will have diagonals $0,k,2k\ldots$ and the folded shapes $\alpha^{(i)} \diamond \alpha^{(k-i)}$ will have diagonals:$\{\ldots i+2k,i+k,k-i,2k-i, 3k-i \ldots \}$ where the specialized diagonal $\mathfrak{d}_i$ will have the diagonal value $k-i$. \label{rmk:diag}
\end{rmk}
Note that the diagonal values $i < (k-1)/2$ do not appear. The reason of this is our convention of setting the head of double ribbon to be the head of the larger piece.

\begin{cor} A $k$-ribbon $R$ has a box on the main diagonal of $\lambda$ if and only if $diag(R)\leq k$.
\end{cor}

\begin{defn}Consider a semi standard $k$-ribbon tableaux $T$ of shape $\lambda$, with $|\lambda|=n$. The \emph{standardization} of $T$, denoted $St(T)$ is the standard $k$-ribbon tableaux of shape $\lambda$ that we obtain by the following numbering:
\begin{itemize}
    \item We number the cells in the order $1'<1<2'<2<\cdots $
   \item If there is more than one cell of label $i$, we order them so that their diagonal values will be increasing.
   \item If there is more than one cell of label $i'$, we order them so that their diagonal values will be decreasing.
\end{itemize}
\label{def:standardization}
\end{defn}

\begin{prop}\label{prop:standardization}$St(T)$ is well defined. 
\end{prop}
\begin{proof} As ribbons labeled $i$ form a horizontal strip, they can be removed in the order diagonals are increasing. Similarly, tibbons labeled $i'$ form a vertical strip and can be removed in the order their diagonals are decreasing. \end{proof}
 
Using the labeling of the diagonals from Remark \ref{rmk:diag}, we can also do the same standardization operation on the semi-standard fillings of the $k$-quotient.

\begin{prop}
 We can extend $\Phi^k_{\lambda}$ to a weight preserving bijection between semi-standard $k$-ribbon tableaux of shape $\lambda$, and semi-standard fillings of its $k$-quotient.
\end{prop}
\begin{proof} Let $T$ be a semi-standard $k$-ribbon tableaux of shape $\lambda$, given by the sequence $\lambda_0 \subset \lambda_{1'}\subset \lambda_{1} \subset \lambda_{2'} \subset \lambda_2 \subset \cdots  \subset \lambda_t=\lambda$ of shifted diagrams. As our definition of standardization respects the inclusion order, $St(T)$ restricted to any $\lambda_{i}$ gives a standardization of $\lambda_i$. The same is true for $\lambda_{i'}$s. 
Let us apply the $\Phi^k_{\lambda}$ to the standardization of $T$. This gives us a bijection $\phi$ between ribbons of $T$ and the boxes on the $k$-quotient. This also can be restricted to the subdiagrams $\lambda_i$ and $\lambda_i'$, giving a sequence  $\Phi^k(St(\lambda_0)) \subset \Phi^k(St(\lambda_{1'}))\subset \Phi^k(St(\lambda_{1})) \subset \cdots  \subset \Phi^k(\lambda_t)=\Phi^k(\lambda)$. Here, the subset relation is defined pointwise in the $(k+1)$-tuples of quotient diagrams. 
\begin{claimm} The filling of the $k$-quotient obtained by this is a semi-standard filling, and is equal to $\Phi^k_{\lambda}(T)$ if the filling $T$ is standard. 
\end{claimm}
\begin{claimmproof} The second part of the claim follows from the definition of $\Phi^k_{\lambda}(T)$. For the first part, we need to show that the filling of each $a^{(i)} \diamond a^{(k-i)}$ gives a semi-standard folded shape and the filling of $a^{(k)}$ gives a semi-standard shifted shape. Let us look first at the case of $a^{(k)}$.
 To obtain a contradiction, let us assume there are two boxes $B_1$ and $B_2$ on $a^{(k)}$ that are marked $j$ and are on the same column (so that they do not form a horizontal $1$-strip). Without loss of generality, we can take $B_2$ to be the higher one. This means if we name the corresponding ribbons on $\lambda$ respectively $R_1$ and $R_2$, we have  $diag(R_2) < diag(R_1)$. Also, as they both are labeled $j$, they are on a horizontal $k$-strip, specifically $H(R_2)$ lies strictly to the right of $H(R_2)$. These together imply that $H(R_2)$ must also be strictly above $H(R_1)$. Remember that the cells labeled $j$ form a skew shape, so the cell $C$ that is on the same row as $H(R_1)$ and the same column as $H(R_2)$ must also be in $\lambda_i$ with its diagonal value higher than those of $H(R_1)$ and $H(R_2)$. This implies it is not on $R_1$ or $R_2$. It must be on a different ribbon $R_3$ on the horizontal $k$-strip. As $R_2$ and $R_3$ have boxes in the same column with the box of $R_2$ above, we can not remove $R_3$ before $R_2$. This implies $H(R_3)$ must be strictly to the right of H($R_2$) and consequently to the right of $C$. This can not happen as $C$ is on $R_3$, by \ref{prop:head}. Symmetrically, no two cells marked $j'$ can be on the same row, so we indeed have a semi-standard filling of $a^{(k)}$. Now consider the boxes marked $j$ on $a^{(i)} \diamond a_{(k-i)}$ in some $i$. As they come from the difference $\Phi^k(St(\lambda_{j})) \backslash \Phi^k(St(\lambda_{j'}))$, they form a skew shape. Also, the boxes that are labeled $j$ to the right of the main diagonal form a horizontal strip by the same reasoning in the case of $\lambda$. The boxes labeled $j$ to the left of the main diagonal form a vertical strip, as we have the inverted version of the same rules. The $j'$ case is again symmetrical.
\end{claimmproof}

Now let us define the inverse of this operation. Given a semi-standard filling $\bar{T}$ of the $k$-quotient, as the $k$-quotient has the diagonal values induced by $\lambda$, we can apply the same standardization algorithm to the quotient, to get a standard filling $St(\bar{T})$ of the quotient. Applying $\Phi^{k-1}_{\lambda}$ to this filing gives a standard filling of $\lambda$. We can use this bijection between cells of the quotient and ribbons of $\lambda$ to carry the labels in $\bar(T)$ to the corresponding ribbons in $\lambda$. Note that, this inverts the above operation by definition.
\begin{claimm} The inverse operation takes $\bar{T}$ to a semi-standard filling of $\lambda$.
\end{claimm}
\begin{claimmproof} Let $R$ and $S$ be two ribbons marked $j$ on $\lambda$. We will show that they form an horizontal strip. The case of $j'$ is symmetrical. First note that we can not have $diag(R)=diag(S)$, as that would imply the corresponding cells in the quotient are both in the same $a^{(i)} \diamond a^{(k-i)}$ (or both in $a^{(k)}$) on the same diagonal, which is not possible. Let us assume, without loss of generality, that  $diag(R) > diag(S)$.  Then, in the standardization, the label of $R$ will be higher than the label of $S$, meaning $R$ can be removed before $S$: $H(R)$ can not be below $H(S)$ in the same column. In this case $H(R)$ is strictly to the right of $H(S)$ implying they form a horizontal strip, as otherwise $H(R)$ would be strictly below $H(S)$ in a row strictly to the left, giving us no possible way to label the ribbon containing the cell $C$ in the same row as $H(R)$ and the same column as $H(S)$.
\end{claimmproof}
\end{proof}
This bijective relationship shows that the $k$-ribbon Q-function is equal to the product of the Q-functions of its quotient:
\begin{thm} The $k$-ribbon Q-function of a shifted shape $\lambda$ with $k$-abacus representation $(\alpha^{(1)},\alpha^{(2)},\ldots ,\alpha^{(k)})$ has the following expansion in terms of Schur's $Q$-functions:
\begin{eqnarray} RQ^{(k)}_{\lambda}(X)=Q_{a^{(k)}}(X)\prod_{i\leq\lfloor k/2\rfloor}Q_{\mu_i}(X) \label{eq:5} \end{eqnarray}
where $\mu_i$ is the underlying skew-unshifted shape of $a^{(i)} \diamond a^{(k-i)}$ if $i<k/2$ and $a^{(k/2)}\diamond \varnothing$ if $i=k/2$. 
\end{thm}

\begin{cor} The Q-ribbon functions expand positively into Schur's Q-functions.
\end{cor}
\begin{proof} Follows from the last theorem and the Schur Q-positivity of the skew-shifted Schur Q-functions \ref{thm:Stembridge}.
\end{proof}

Note that the Schur's Q-functions are themselves $k$-ribbon Q-functions for any $k$, as  $k\lambda=(k\lambda_1,k\lambda_2,\ldots,k\lambda_n)$ has $RQ^{(k)}_{k\lambda}(X)=Q_{\lambda}$.

\section{Peak Functions of Ribbon Tableaux}

The \emph{reading word} of a $k$-ribbon tableau is a reading of the labels on the heads of the ribbons, left to right, top to bottom.
\begin{defn}
A \emph{marked standard shifted $k$-ribbon tableau} $T'$ of shape $\lambda$ is defined to be a standard shifted $k$-ribbon tableau $T$ of shape $\lambda$ together with a subset $M$ of $[n]$ determining the marked coordinates. On the Young diagram, for all $i$ in M, we replace the label of $R_i$ with $i'$. We will denote the set of all the marked versions of $T$ by the set $Mark(T)$.
\end{defn}

\begin{figure}[ht]
\centering
\begin{tikzpicture}[scale=3/8]
\draw[thick] (0,0)--(1,0)--(1,-1)--(2,-1)--(2,-4)--(1,-4)--(1,-2)--(0,-2)--(0,0) ;

\draw[thick] (0,-2)--(0,-5)--(3,-5)--(3,-4)--(2,-4);

\draw[thick] (-2,-3)--(-3,-3)--(-3,-4)--(-4,-4)--(-4,-5)--(0,-5);

\draw[thick] (0,-1)--(-1,-1)--(-1,-2)--(-2,-2)--(-2,-3)--(-2,-4)--(0,-4);
\node at (1.5,-3.45) {4};
\node  at (2.5,-4.5) {$3'$};
\node  at (-0.5,-3.4) {$2'$};
\node at (-0.5,-4.5) {1};
\end{tikzpicture}
\caption{This marked $5$-ribbon tableau has reading word $2',4,1,3'$ }
\end{figure}
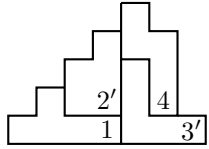

\begin{thm}\label{thm:peak} The $k$-ribbon Q-function of a shifted shape $\lambda$ we have can be written in terms of descent functions and peak functions as follows:
\begin{eqnarray*}RQ^{(k)}_{\lambda}(X)=\sum_{T'\in SShT\pm^{(k)}(\lambda)}F_{Des(T')}\qquad RQ^{(k)}_{\lambda}(X)=\sum_{T\in SShT^{(k)}(\lambda)}2^{|Peak(Des(T))|+1}G_{Peak(Des(T))}\end{eqnarray*}
where $SShT\pm^{(k)}(\lambda)$ is the set of marked standard shifted $k$-ribbon tableaux of shape $\lambda$.
\end{thm}

A \emph{run} of a subset $D$ of $[n]$ is a maximal subset of consecutive numbers. We will denote by $Run(D)$ the set of the runs of $D$. Note that $\dot{\bigcup}_{Run(D)}=D$. For example, $D=\{2,3,5,8,9,10\}=\{2,3\} \cup \{5\} \cup \{8,9,10\}$.

\begin{prop} We can calculate the number of the peaks of a tableaux $T$ from its descent set as follows:
$$|P|=\begin{cases}
|Run(Des(T))|&  1\notin Des(T) \\
|Run(Des(T))| - 1 & 1\in Des(T)\end{cases}$$
\end{prop}
\begin{proof} This follows from the fact that the elements $j$ of the peak set satisfy $j \in D$ and $j-1\notin D$ for all $j >1$, so that elements of the peak set are given by the smallest elements of each run, with the exception of the case if there is a run starting with $1$.
\end{proof} 

\begin{lem} For any $T' \in$ $Mark(T)$, the descent set of $T'$ is independent of whether a given $i\leq n$ is marked if and only if:
\begin{itemize}
    \item $i>1$ with $i-1 \in Des(T)$, $i \notin Des(T)$ or
    \item $i=1 \notin Des(T)$ 
\end{itemize}
The number of such $i$ is given by $|Peak(T)|-1$.
\label{lemma:numberofi}
\end{lem}
\begin{proof} The first part comes from Lemma \ref{lem:7}. Also, a number $i$ satisfies these conditions iff $i$ is one lower than the lowest number of a run of $Des(T)$. That means, if $1\notin D$, it will be equal to the number of runs of $Des(T)$. If $1\in T$, it will be $|Run(Des(T))|-1$.
\end{proof}
\begin{prop}  For any $T' \in$ $Mark(T)$, we have :
$$Spike(T') \supset Peak(T)$$
\end{prop}
\begin{proof} Note that if $i\in Peak(Des(T))$, we have $i \in Des(T)$ and $i-1 \notin Des(T)$. For any given $T'$ if $i$ is unmarked on $T'$, then $i \in Des(T')$ and $i-1 \notin Des(T')$ so $i \in Spike(T')$. Otherwise $i$ is marked, so that we have $i \notin Des(T')$ and $i-1 \in Des(T')$, implying again that $i \in Spike(T')$.
 \end{proof}

The proposition above shows that the descent map takes the elements of $Mark(T)$ to subsets $D$ of $[n-1]$ with $Spike(D) \subset Peak(T)$. Next, we will show that this map is surjective. In fact, we will prove the stronger statement that the preimage of every element is of the same size.

\begin{lem} Assume $D$ is a subset of $[n-1]$ satisfying $Spike(D) \supset Peak(T)$. Then, there is a marked version $T'$ of $T$ such that $Des(T')=D$.
\label{lem:spikepeak}
\end{lem}
\begin{proof} Let us generate a marked version $T'$ of $T$ as follows: Starting with $i=1$, at Step $i$ we mark $i$ if $i\in Des(T)$ and $i \notin D$, and we mark $i+1$ if $i \notin Des(T)$ and $i \in D$ (marking the same number a second time has no effect). Then we move on to the next number, till we go through all $i\leq n-1$.

Let us verify that the descent set of $T'$ is indeed equal to $D$. For a fixed $i$ assume $i\in Des(T)$. Then by Lemma \ref{lem:7} $i$ is a descent of $T'$ iff $i$ is unmarked. Therefore, it is sufficent to show $i$ is unmarked iff $i\in D$. If $i \notin D$, then we marked $i$ onStep $i$, so $i \in Des(T')$. Otherwise $i \in D$, and we can only have marked $i$ at step $i-1$. This implies $i-1 \notin Des(T)$ and $i-1 \in D$. This contradicts our assumption $Spike(D) \supset Peak(T)$ as $i$ is a peak of $T$ but not a spike of $T$. The case $i\notin Des(T)$ is similar.
\end{proof}

\begin{prop} The descent map taking elements of  $Mark(T)$ to subsets $D$ of $[n-1]$ with $Spike(D) \supset Peak(Des(T))$ is a $2^m$ to one cover, where $m=|Peak(Des(T))|+1$
\label{prop:kor}
\end{prop}
\begin{proof} The number of subsets $D$ of $[n-1]$ with $Spike(D) \supset Peak(Des(T))$ is given by $2^{n-1 -|Peak(Des(T))|}$. By Lemma \ref{lem:spikepeak}, we know that the descent map is surjective. By Lemma \ref{lemma:numberofi}, the preimage of each element under the descent map contains at least $2^m$ elements. $2^{m}\times 2^{n-1 -|Peak(Des(T))|}=2^n$ which is the cardinality of $Mark(T)$, so the preimage of each element must contain exactly $2^m$ elements.
\end{proof}

Now we are ready to prove Theorem \ref{thm:peak} from the beginning of the section.

\begin{proof}[Proof of  Theorem \ref{thm:peak}]
 Let $S$ be a semi-standard $k$-ribbon tableaux of shape $\lambda$. We have already defined the standardization of $S$ (Definition \ref{def:standardization}). Let us denote by $St'(S)$ the marked standardization of $S$, which is simply standardization while keeping the marked cells marked. 
We will show that there is a bijection $\phi_{T'}$ between semi-standard $k$-ribbon tableaux $S$ that standardize to $T'$ and the combinations refining $Des(T')$, satisfying $x^{|S|}=x^{|\phi_{T'}(S)|}$.
This will imply:
$$\sum_{T'\in SShT\pm^{(k)}(\lambda)}F_{Des(T')}(X)=\sum_{T'}\sum_{C\in Ref(Des(T'))}X^C=\sum_{T'}\sum_{C}x^{|\phi_{T'}^{-1}(C)|}=\sum_{S\in SSShT^{(k)}(\lambda)}X^{|S|}=RQ^{(k)}_{\lambda}(X)$$
where for a combination $C=(c_1,c_2,\ldots,c_t)$ we use $X^C$ to denote $x_1^{c_1}x_2^{c_2}\cdots x_t^{c_t}$.

Assume $S$ satisfies $St'(S)$=$T'$. We define $\phi_{T'}(S)=(i_1,i_2.. )$ where $i_m$ stands for the total number of cells labelled $m$ or $m'$ on $S$. As $S$ has $n$ ribbons, $\phi_{T'}(S)$ will be a combination of $n$ that satisfies $x^{|S|}=x^{|\phi_{T'}(S)|}$.

\begin{claimm} $S$ refines $Des(T')$.
\end{claimm}
\begin{claimmproof} Let $S$ be a semi-standard filling with $St'(S)=T'$. Consider the pre-image of ribbon $R_i$ (the unique ribbon labeled $i$ or $i'$ on $T'$) under $St'$. We will denote the label of this ribbon in S by $St'^{-1}(i)$. To prove that $S$ refines $Des(T')$, it is sufficient to show that if $i$ is a descent of $T'$, then $St'^{-1}(i)$ and $St'^{-1}(i+1)$ are not both elements of $\{m,m'\}$ for any $m$ (Note that, by the standardization algorithm, we will have $St'^{-1}(i)\leq St'^{-1}(i+1)$in any case). 

Let $i$ be a descent of $T'$. By Lemma \ref{lem:7}, there are two possiblities:
\begin{itemize}
    \item Case 1, $i \in Des(T)$ and $i$ is not marked in $T'$: Then $St'^{-1}(i)$ is an unmarked number $m$. $St'^{-1}(i) \geq m$, so it can not be $m'$. Assume it is also $m$. Then, we have two ribbons labeled $m$, but $diag(R_i) > diag(R_{i+1})$ by the definition of standardization (Definition \ref{def:standardization}). 
    \item Case 2, $i \notin Des(T)$ and $i+1$ is marked in $T'$: This means $St'^{-1}(i+1)$ is a marked number $m'$, and $St'^{-1}(i) \geq m'$ can not be $m$. It can not be $m'$ either, because as in the first case, we get two ribbons labeled $m'$ but $diag(R_i) < diag(R_{i+1})$, which as in Case 1, contradicts the definition of standardization.
\end{itemize} \end{claimmproof}
\begin{claimm} For any combination $C$ refining $Des(T)$ there is a unique $S$ that standardizes to $T'$ with $\phi_{T'}(S)=C$
\end{claimm}
\begin{claimmproof} Let $C=(c_1,c_2,\ldots,c_T)$ be a combination of $n$ that refines $Des(T')$. We will define $S$ by labeling the ribbons $R_1$ to $R_{c_1}$ with $1$, ribbons $R_{c_1+1}$ to $R_{c_1+c_2}$ with $2$, $R_{c_1+c_2+1}$ to $R_{c_1+c_2+c_3}$ with $3$ and so on, and then marking the image of $R_i$ iff $i$ is marked in $T$. We need to show that this $S$ is semi-simple, and it standardizes to $T'$. Uniqueness then, comes from the fact that the placement of the markings are preserved.

Assume ribbons $R_i$ and $R_{i+1}$ have the same unmarked label $m$. Then, $i\notin Des(T')$ and $i$ and $i+1$ are both not labeled in $T'$, so we must have $i \notin Des(T)$ by Lemma \ref{lem:7}. That means, $diag(R_i) < diag(R_{i+1})$, so unmarked numbers are ordered so that their diagonals will increase in $T'$. Similarly, if $R_i$ and $R_{i+1}$ both have the same marked label $m'$, then $i\in Des(T')$ and $diag(R_i)  > diag(R_{i+1})$. These mean that we have $St'(S)=T$.

Additionally, we can remove ribbon $R_{i+1}$ before $R_{i}$. This implies that if both ribbons are labeled $m$, $H(R_{i+1})$ is going to be strictly to the right of $H(R_i)$ as $diag(R_i) < diag(R_{i+1})$. If they are both labeled $m'$, $H(R_{i+1})$ is going to be strictly above $H(R_i)$ as $diag(R_i)  > diag(R_{i+1})$.
\end{claimmproof}

This proves the expansion of the $k$-ribbon function in terms of descent functions. The peak function expansion follows by Proposition \ref{prop:kor}:
$$RQ^{(k)}_{\lambda}(X)=\sum_{T'\in SShT\pm^{(k)}(\lambda)}F_{Des(T')}=\sum_{T}\sum_{T'\in Mark(T)} F_{Des(T')}=\sum_{T}2^{|Peak(Des(T))|+1}G_{Peak(Des(T))}$$
\end{proof}

 \section{Type B LLT polynomials}
 
 In \cite{MR1434225}, Lascoux, Leclerc and Thibon give a $q$-analogue for the $k$-ribbon that is Schur positive functions for the unshifted case . For this, they use the spin statistic on ribbon tableaux which depends on the total height of its ribbons. In this section, we show that there is no direct way of extending the concept of height to double ribbons that will give positive structure coefficients in the shifted case. Nevertheless, we are able to give a non-trivial $q$-analogue for the shifted ribbon functions and prove its Schur Q-positivity.
 
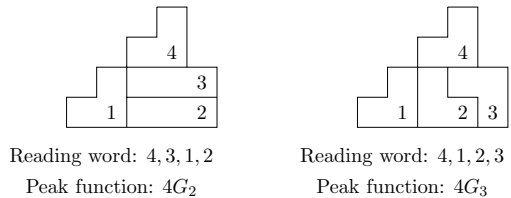
\begin{figure}[ht]
\centering
\scalebox{.7}{
\begin{tikzpicture}[scale=4/7]
\draw (0,0)--(1,0)--(1,-2)--(2,-2)--(2,-4)--(-3,-4)--(-3,-3)--(-2,-3)--(-2,-2)--(-1,-2)--(-1,-1)--(0,-1)--(0,0);
\draw (1,-2)--(-1,-2)--(-1,-4) (-1,-3)--(2,-3);
\node[font=\large] at (0.5,-1.5) {4};
\node[font=\large] at (1.5,-2.5) {3};
\node[font=\large] at (1.5,-3.5) {2};
\node[font=\large] at (-1.5,-3.5) {1};
\node[font=\large] at (-1.5,-5) {Reading word: $4,3,1,2$};
\node[font=\large] at (-1.5,-6) {Peak function: $4G_{2}$};
\end{tikzpicture}\qquad \qquad
\begin{tikzpicture}[scale=4/7]
\draw (0,0)--(1,0)--(1,-2)--(2,-2)--(2,-4)--(-3,-4)--(-3,-3)--(-2,-3)--(-2,-2)--(-1,-2)--(-1,-1)--(0,-1)--(0,0);
\draw (1,-2)--(-1,-2)--(-1,-4) (0,-2)--(0,-3)--(1,-3)--(1,-4);
\node[font=\large] at (0.5,-1.5) {4};
\node[font=\large] at (0.5,-3.5) {2};
\node[font=\large] at (1.5,-3.5) {3};
\node[font=\large] at (-1.5,-3.5) {1};
\node[font=\large] at (-1.5,-5) {Reading word: $4,1,2,3$};
\node[font=\large] at (-1.5,-6) {Peak function: $4G_{3}$};
\end{tikzpicture}}
    \caption{The two $3$-ribbon tableaux of shape $\{5,4,2,1\}$ and their corresponding peak functions}
    \label{fig:cex1}
\end{figure}

\begin{thm} There is no “intrinsic” definition for the height of a double ribbon, which, along with the usual definition of heights for the single ribbon, gives a Schur Q-positive or even a Schur positive function. Here by intrinsic, we mean there is no definition that only comes from the shape of the double ribbon, and is independent of its placement or the other ribbons in the shape.
\end{thm}

\begin{proof} If we consider the example in Fig \ref{fig:cex1}, we can see that the only difference between the two fillings is the placement of ribbons 2 and 3. For any intrinsic definition, the heights of the double ribbons 4 and 1 would match in the two fillings. The total height of 2 and 3 being higher on the shape on the right, we get a function $4q^cG_2+4q^dG_3$ with $c\neq d$ which is not Schur P-positive. In fact, $G_2$ and $G_3$ by themselves are not even Schur positive or symmetric functions.
\end{proof}

\begin{figure}[ht]
    \centering
    \scalebox{.7}{
\begin{tikzpicture}[scale=4/7]
\draw (0,0)--(1,0)--(1,-2)--(2,-2)--(2,-4)--(-3,-4)--(-3,-3)--(-2,-3)--(-2,-2)--(-1,-2)--(-1,-1)--(0,-1)--(0,0);
\draw (1,-2)--(-1,-2)--(-1,-4) (-1,-3)--(2,-3);
\draw (2,-2)--(3,-2)--(3,-3)--(4,-3)--(4,-4)--(2,-4);
\node[font=\large] at (-2.5,-0.5) {$\textrm{I}$};
\node[font=\large] at (0.2,-1.5) {5/4};
\node[font=\large] at (1.5,-2.5) {3};
\node[font=\large] at (1.5,-3.5) {2};
\node[font=\large] at (3.1,-3.5) {4/5};
\node[font=\large] at (-1.5,-3.5) {1};
\node[font=\large] at (0.7,-5) {Reading words: $5,3,1,2,4$};
\node[font=\large] at (3.1,-6) {$4,3,1,2,5$};
\node[font=\large] at (1.04,-7) {Peak function: $8G_{\{2,4\}}+4G_{\{2\}}$};
\end{tikzpicture}\hspace{0.2cm}
\begin{tikzpicture}[scale=4/7]
\draw (0,0)--(1,0)--(1,-2)--(2,-2)--(2,-4)--(-3,-4)--(-3,-3)--(-2,-3)--(-2,-2)--(-1,-2)--(-1,-1)--(0,-1)--(0,0);
\draw (1,-2)--(-1,-2)--(-1,-4) (0,-2)--(0,-3)--(1,-3)--(1,-4);
\draw (2,-2)--(3,-2)--(3,-3)--(4,-3)--(4,-4)--(2,-4);
\node[font=\large] at (-2.5,-0.5) {$\textrm{II}$};
\node[font=\large] at (0.2,-1.5) {5/4};
\node[font=\large] at (0.5,-3.5) {2};
\node[font=\large] at (1.5,-3.5) {3};
\node[font=\large] at (3.1,-3.5) {4/5};
\node[font=\large] at (-1.5,-3.5) {1};
\node[font=\large] at (0.7,-5) {Reading words: $5,1,2,3,4$};
\node[font=\large] at (3.1,-6) {$4,1,2,3,5$};
\node[font=\large] at (1.,-7) {Peak function: $4G_{\{4\}}+4G_{\{3\}}$};
\end{tikzpicture}\hspace{0.2cm}
\begin{tikzpicture}[scale=4/7]
\draw (0,0)--(1,0)--(1,-2)--(2,-2) (2,-4)--(-3,-4)--(-3,-3)--(-2,-3)--(-2,-2)--(-1,-2)--(-1,-1)--(0,-1)--(0,0);
\draw (1,-2)--(-1,-2)--(-1,-4) (0,-2)--(0,-3)--(1,-3)--(1,-4);
\draw (2,-2)--(3,-2)--(3,-3)--(4,-3)--(4,-4)--(2,-4);
\draw (1,-3)--(3,-3);
\node[font=\large] at (-2.5,-0.5) {$\textrm{III}$};
\node[font=\large] at (0.5,-1.5) {5};
\node[font=\large] at (0.5,-3.5) {2};
\node[font=\large] at (2.5,-2.5) {4};
\node[font=\large] at (3.5,-3.5) {3};
\node[font=\large] at (-1.5,-3.5) {1};
\node[font=\large] at (0.7,-5) {Reading word: $5,4,1,2,3$};

\node[font=\large] at (1,-7) {Peak function: $4G_{\{3\}}$};
\end{tikzpicture}}
    \caption{Standard $3$-ribbon tableaux of shape $(7,5,2,1)$ and their corresponding peak functions}
    \label{fig:cex2}
\end{figure}
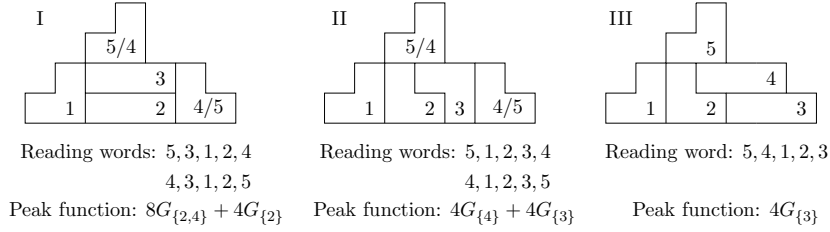
 
 Another example  where all the tableaux need to have the same cospin to obtain Schur Q-positivity is given in Figure \ref{fig:cex2}. A common point of these two examples with trivial cospin is that both have only one piece in their $3$-quotient, which motivates a slightly technical $q$-analogue defined through the quotient.

\begin{defn} For a shifted shape $\lambda$ with $k$-abacus representation $(\alpha^{(1)},\alpha^{(2)},\ldots,\alpha^{(k)})$, we define the $q$-analogue of the shifted $k$-ribbon Q-function as follows:
\begin{eqnarray}
QR^{(k)}_{\lambda}(X;q):= Q_{\alpha^{(k)}}(X) \sum_{T \in SRT_{\lfloor k/2 \rfloor}(\mu)} q^{cospin (T)}2^{|Peak(T)|+1} F_{Peak(T)}(X)
\end{eqnarray} where $\mu$ is the unshifted partition corresponding to the $\lfloor k/2 \rfloor$-quotient $(\mu^1,\mu^2,\ldots,\mu^{\lfloor k/2 \rfloor})$, with $\mu^i=a^{(i)} \diamond a^{(k-i)}$ if $i<k/2$ and $\mu_{k/2}=a^{(k/2)}\diamond \varnothing$ when $k$ is even.
\end{defn}

Note that when $q=1$, we get the formulation of $QR^{(k)}_{\lambda}(X)$ given in Equation $\ref{eq:5}$, so $QR^{(k)}_{\lambda}(X;1)=QR^{(k)}_{\lambda}(X)$ as desired.
\begin{thm} \label{thm:last}The function $QR^{(k)}_{\lambda}(X;q)$ has an expansion into Schur's Q-functions with coefficients from $\mathbb{Z}^+[q]$.
\end{thm}
\begin{proof} Let $(\alpha^{(1)},\alpha^{(2)},\ldots,\alpha^{(k)})$ be the $k$-abacus representation for $\lambda$, and $\mu$ is the be the unshifted partition corresponding to the $\lfloor k/2 \rfloor$-quotient $(\mu^1,\mu^2,\ldots,\mu^{\lfloor k/2 \rfloor})$ as defined in Theorem \ref{thm:last}. The LLT polynomial for $\mu$ satisfies

\begin{eqnarray*}
GF^{(\lfloor k/2 \rfloor)}_{\mu/\mu_0}(X;q)=\sum_{T \in SRT_{\lfloor k/2 \rfloor}(\mu)} q^{cospin (T)} F_{Des(T)}(X)=\sum f_{\gamma}^{\mu}(q) s_{\gamma}(X).
\end{eqnarray*} where $\gamma$ are unshifted shapes and $f_{\gamma}^{\mu}(q)$ have positive integer coefficients. As any unshifted $\gamma$ can be seen as a skew-shifted shape $\gamma^+/\delta_{\ell(\gamma)}$. Multiplying by $2^{|Peak(T)|+1} F_{Peak(T)}(X)$ instead $ F_{Des(T)}(X)$ for each $\gamma$ tableau $T$ corresponds to calculating $Q_{\gamma^+/\delta_{\ell(\gamma)}}(X)$ instead of $s_{\gamma}(X)$. So we have:
\begin{eqnarray*}
QR^{(k)}_{\lambda}(X;q):= Q_{\alpha^{(k)}}(X)\big(\sum_{T \in SRT_k(\mu)} q^{cospin (T)} 2^{|Peak(T)|+1} F_{Peak(T)}(X) \big)= Q_{\alpha^{(k)}}(X)\big(\sum f_{\gamma}^{\mu}(q) Q_{\gamma+/\delta_{\ell(\gamma)}}(X)\big)
\end{eqnarray*}
As multiplication and skewing operations on Schur's Q-functions give positive expansions into Schur's Q-functions the result follows. 
\end{proof}

Let us finish with calculating an example.

Consider the shape $(9,8,6,2)$ with the $5$-ribbon quotient $\{(2,1),(2),\varnothing\}$ with standard fillings given in Figure \ref{fig:fullpage}. It has the Q $5$-ribbon function:
\begin{eqnarray*}
QR^{(5)}_{(9,8,6,2)}(X)= Q_{(3,1)/(1)}(X)~Q_{(3)/(1)}(X)= 2Q_{(5)}(X)+4Q_{(4,1)}(X)+3Q_{(3,2)}(X)
\end{eqnarray*} 

Viewed as a $2$-quotient, $\{(2,1),(2)\}$ corresponds to unshifted shape $(4,4,1,1)$ with the LLT polynomial:

\begin{eqnarray*}
F^{2}_{(4,4,1,1)}(X;q)&=&\sum_{T \in SRT_2(4,4,1,1)} q^{cospin (T)} F_{Des(T)}(X)=q^2s_{(2,2,1)}(X)+qs_{(3,1,1)}(X)+qs_{(3,2)}(X)+s_{(4,1)}(X)
\end{eqnarray*} Viewing $(2,2,1)$, $(3,1,1)$, $(3,2)$ and $(4,1)$ as skew shifted shapes, we get the following:
\begin{eqnarray*}
QR^{(5)}_{(9,8,6,2)}(X;q)&=& P_{\varnothing}(X)\big(q^2Q_{(5,3,2)/(2,1)}(X)+qQ_{(5,4,1)/(2,1)}(X)+ qQ_{(4,2)/(1)}(X)+Q_{(5,1)/(1)}(X) \big)\\
&=& (q+1)Q_{(5)}(X)+(q^2+2q+1)Q_{(4,1)}(X)+(q^2+2q)Q_{(3,2)}(X)
\end{eqnarray*}

\section*{Acknowlegdements}
The author would like to thank Prof. Sami Assaf for valuable direction and encouragement throughout this project.

\bibliography{main}
\bibliographystyle{plain}
\bigskip

\end{document}